\documentclass[final]{siamltex}
\usepackage{amsmath, amsfonts, amssymb, epsfig}
\usepackage{color}
\usepackage{url}

\def\e{{\epsilon}}

\def\0{{\bf 0}}
\def\1{{\bf 1}}
\def\e{{\bf e}}

\def\R{\mathbb{R}}

\providecommand{\com[2]}{\begin{tt}[#1: #2]\end{tt}}

\def\ao#1{{#1}}
\def\aor#1{{#1}}

\newtheorem{remark}[theorem]{\indent {Remark}}
\usepackage{tikz, tkz-berge}

\title{Cooperative learning in multi-agent systems from intermittent measurements\thanks{This work was supported in part by AFOSR grant FA9550-07-1-0-0528, ONR grant N00014-09-1-1074 and ARO grant W911NG-11-1-0385.}}

\author{Naomi Ehrich Leonard\thanks{Department of Mechanical and Aerospace Engineering, Princeton University, Princeton, NJ, 08544, USA ({\tt naomi@princeton.edu}).}
        \and Alex Olshevsky\thanks{Department of Industrial and Enterprise Systems Engineering, University of Illinois at Urbana-Champaign, Urbana, IL, 61801, USA ({\tt aolshev2@illinois.edu}).  }}

\begin{document}

\maketitle

\begin{abstract} Motivated by the problem of tracking a direction in a decentralized way, we consider the general problem of cooperative learning in multi-agent systems with time-varying connectivity and intermittent measurements. We propose a distributed learning protocol capable of learning an unknown vector $\mu$ from noisy measurements made independently by autonomous nodes. Our protocol is completely distributed and  able to cope with the time-varying, unpredictable, \aor{and noisy} nature of inter-agent communication, and intermittent noisy measurements of $\mu$. Our main result bounds the learning speed of our protocol in terms of the size and combinatorial features of the (time-varying) networks connecting the nodes. 
\end{abstract}



\pagestyle{myheadings}
\thispagestyle{plain}

\section{Introduction\label{intro}}

Widespread deployment of mobile sensors is expected to revolutionize
our ability to monitor and control physical environments. However, for these networks 
to reach their full range of applicability they must be capable of
operating in uncertain and unstructured environments. Realizing
the full potential of networked sensor systems will require the development of protocols that are fully
distributed and adaptive in the face of persistent faults and time-varying,
unpredictable environments. 

\ao{Our goal in this paper is to initiate the study of cooperative multi-agent learning by distributed
networks operating in unknown and changing environments, subject to faults and failures of 
communication links. While our focus here is on the basic problem of learning an unknown vector, 
we hope to contribute to the development of a broad theory of cooperative, distributed learning in such environments, with 
the ultimate aim of designing sensor network protocols capable of adaptability. }

\ao{We will study a simple, local protocol for learning a vector from intermittent
measurements and evaluate its performance in terms of the number of nodes and the (time-varying) network structure.} Our
 direct motivation is the problem of tracking a direction from chemical gradients. A network of 
mobile sensors needs to move in a direction $\mu$ (understood as a vector on the unit circle), which none 
of the sensors initially know; however, intermittently some sensors are able to obtain a sample of $\mu$. The sensors
can observe the velocity of neighboring sensors but, as the sensors move, the set of neighbors of each sensor changes.  The challenge
is to design a protocol by means of which the sensors can adapt their velocities based on the measurements 
of $\mu$ and observations of the velocities of neighboring sensors so that every node's velocity converges to $\mu$ as 
fast as possible. This challenge is further complicated by the fact that all estimates of $\mu$ as well as all 
observations of the velocities of neighbors
are assumed to be noisy.

We will consider a natural generalization in the problem, \ao{wherein we abandon the constraint that $\mu$ lies on the
unit circle} and instead consider the problem of learning an arbitrary vector $\mu$ by a network of mobile nodes subject to time-varying (and unpredictable) inter-agent connectivity, and intermittent, noisy measurements. \ao{We will be interested in the speed at which
local, distributed protocols are able to drive every node's estimate of $\mu$ to the correct value. We will be especially concerned with identifying the salient 
features of network topology that result in good (or poor) performance.}   

\subsection{Cooperative multi-agent learning}

We begin by formally stating the problem. We consider $n$ autonomous nodes engaged in the task of learning a vector $\mu \in \mathbb{R}^l$. At each time  $t=0,1,2,\ldots$ we denote by $G(t)=(V(t),E(t))$ the graph of inter-agent communications at time $t$:
two nodes are connected by an edge in $G(t)$ if and only if they are able to exchange messages at time $t$. Note that by definition the graph $G(t)$ is undirected. If $(i,j) \in G(t)$ then we will say that $i$ and $j$ are neighbors at time $t$. We will adopt the convention that $G(t)$ contains no self-loops. We will assume the graphs $G(t)$ satisfy a standard condition of uniform connectivity over a long-enough 
time scale: namely, there exists some 
constant positive integer $B$ (unknown to any of the nodes) such that the graph sequence $G(t)$ is $B$-connected, i.e. the graphs $(\{1,\ldots,n\}, \bigcup_{kB+1}^{(k+1)B} E(t) )$ are connected for each
integer $k \geq 0$. \ao{Intuitively, the uniform connectivity condition means that once we take all the edges that have appeared between times $kB$ and $(k+1)B$, the graph is connected}.

Each node maintains an estimate of $\mu$; we will denote the estimate of  node $i$ at time $t$ by $v_i(t)$. At time $t$, node $i$ can update $v_i(t)$ as a function of the noise-corrupted estimates $v_j(t)$ of its neighbors. We will use $o_{ij}(t)$ to denote the noise-corrupted estimate of the offset $v_j(t)-v_i(t)$ available to neighbor $i$ at time $t$: \[ o_{ij}(t) = v_j(t) - v_i(t) + w_{ij}(t) \] Here $w_{ij}(t)$ is a zero-mean random vector every entry of which has variance $(\sigma')^2$, and all $w_{ij}(t)$ are assumed to be independent of each other, as well as all other
random variables in the problem (which we will define shortly). These updates may be the result of a wireless message exchange or may come about as a result of sensing by each node. Physically, each node is usually able to  sense  (with noise) the relative difference $v_j(t) - v_i(t)$, for example if $v_i(t)$ represent 
velocities and measurements by the agents are made in their frame of reference. Alternatively, it may be that nodes are able to measure the absolute 
quantities $v_j(t), v_i(t)$ and then $w_{ij}(t)$ is the sum of the noises in these two measurements.  

 Occasionally, some nodes have access to a noisy measurement \[ \mu_i(t) = \mu + w_i(t),\] where $w_i(t)$ is a zero-mean random vector every entry of which has variance $\sigma^2$; we assume all vectors $w_i(t)$ are independent of each other and of all $w_{ij}(t)$. In this case, node $i$ incorporates this measurement into its updated estimate $v_i(t+1)$. We will refer to a time $t$ when at least one node has a measurement as a {\em measurement time}. For the rest of the paper, we will be making an assumption of uniform measurement speed, namely that fewer than $T$ steps pass between successive measurement times; more precisely, letting $t_k$ be the times when at least one node makes a measurement, we will assume that $t_1 = 1$ and $|t_{k+1} - t_k|  < T$ for all positive integers $k$. 

It is useful to think of this formalization in terms of our motivating scenario, which is a collection of nodes - vehicles, UAVs, mobile 
sensors, or underwater gliders - which need to learn and follow a direction. Updated information about the direction arrives from time to time as one or more of the nodes takes measurements, and the nodes need a protocol by which they update their velocities $v_i(t)$ based on the measurements and observations of the velocities of neighboring nodes. 

This formalization also describes the scenario in which a moving group of animals  must all learn which way to go based on intermittent samples of a preferred direction and social interactions with near neighbors.  An example is collective migration where high costs associated with obtained measurements of the migration route suggest that the majority of individuals rely on the more accessible observations of the relative motion of their near neighbors when they update their own velocities \cite{gc10}.

\subsection{Our results\label{section:controllaw}\label{section:results}} We now describe the protocol which we analyze for the remainder of
this paper. If at time $t$ node $i$ does not have a measurement of $\mu$, it nudges its velocity in the direction of its neighbors:
\begin{equation} \label{nonmeasuringupdate} v_i(t+1) = v_i(t) + \frac{\Delta(t)}{4} \sum_{j \in N_i(t)} \frac{\aor{o_{ij}(t)}}{\max(d_i(t),d_j(t))},
\end{equation} where $N_i(t)$ is the set of neighbors of node $i$ at time $t$, $d_i(t)$ is the cardinality of $N_i(t)$, and $\Delta(t)$ is a stepsize 
which we will specify later. 

On the other hand, if node $i$ does have a measurement $\mu_i(t)$, it updates as
\begin{equation} \label{measuringupdate} v_i(t+1) = v_i(t) + \frac{\Delta(t)}{4} \left( \mu_i(t) - v_i(t) \right) +  \frac{\Delta(t)}{4} \sum_{j \in N_i(t)} \frac{\aor{o_{ij}(t)}}{\max(d_i(t), d_j(t))}. \end{equation} 

Intuitively, each node seeks to align its estimate $v_i(t)$ with both the measurements it takes and estimates of neighboring nodes. As nodes align with one another, information from each measurement slowly propagates throughout the system. 

Our protocol is motivated by a number of recent advances within the literature on multi-agent consensus. On the one hand, the weights we accord to neighboring nodes are based on Metropolis weights (first introduced within the context of multi-agent control in \cite{bdx04}) and are chosen because they lead to a tractable Lyapunov analysis as in \cite{noot09}. On the other hand, we introduce a stepsize $\Delta(t)$ which we will later choose to decay to zero with $t$ at an appropriate speed by analogy with the recent work on multi-agent optimization 
\cite{no09, sn11, yns12}.

\aor{ The use of a stepsize $\Delta(t)$ is crucial for the system to be able to successfully learn the unknown vector $\mu$ with this scheme. Intuitively, as $t$ gets large,  the nodes should avoid  overreacting by changing their 
estimates in response to every new noisy sample. Rather, the influence of every new sample on the estimates $v_1(t), \ldots, v_n(t)$ should decay with $t$: the
more information the nodes have collected in the past, the less they should be inclined to revise their estimates in response to a new sample. This is accomplished by ensuring that the influence of each successive new sample decays with the stepsize $\Delta(t)$.}

We note that our protocol is also motivated by models used to analyze collective decision making and collective motion in animal groups \cite{c05,lsnscl12}.  Our time varying stepsize rule is similar to models of context-dependent interaction in which individuals reduce their reliance on social cues when they are progressing towards their target \cite{tnc09}.


 We now proceed to set the background for our main result, which bounds the rate at which the estimates $v_i(t)$ converge to $\mu$. We first state a proposition which assures us that the estimates $v_i(t)$ do indeed converge to $\mu$ almost surely. 

\bigskip

\begin{proposition} \label{thm:conv} If the stepsize $\Delta(t)$ is nonnegative, nonincreasing and satisfies 
\[ \sum_{t=1}^{\infty} \Delta(t) = +\infty, ~~~~~~~ \sum_{t=1}^{\infty} \Delta^2(t) < \infty,  ~~~~~~~~~ \sup_{t \geq 1} \frac{\Delta(t)}{\Delta(t+c)} < \infty ~~~~ \mbox{ for any fixed integer } c \] then for any initial values
$v_1(0), \ldots, v_n(0)$, we have that with probability $1$ \[ \lim_{t \rightarrow \infty} v_i(t) = \mu ~~~~ \mbox{ for all } i. \]
\end{proposition}


We remark that this proposition may be viewed as a generalization of earlier results on leader-following and learning, which achieved similar conclusions either without the
assumptions of noise, or on fixed graphs, or with the assumption of a fixed leader (see \cite{jlm, leader1, leader2, leader3, angelia-asu-learn, bfh, moura7} as well as the related \cite{hj, carli2}). Our protocol is very much in the spirit of this earlier literature. All the previous protocols (as well as ours) may be thought of as consensus protocols driven by inputs, and we note there are a number of other possible variations on this 
theme which can accomplish the task of learning the unknown vector $\mu$.  

\ao{Our main result in this paper is a strengthened version of \aor{Proposition} \ref{thm:conv} which  provides quantitative bounds on the rate at which convergence to $\mu$ takes place. We are particularly interested in the scaling of the convergence time with the number of nodes and with the combinatorics of the interconnection graphs $G(t)$. We will adopt the natural measure of how far we are from convergence, namely} the sum of the squared distances from the final limit:  \[ Z(\aor{t}) = \sum_{i=1}^n ||v_i(t) - \mu||_2^2. \] We will refer to $Z(t)$ as the variance at time $t$. 

Before we state our main theorem, we introduce some notation. First, we define the the notion of the very lazy Metropolis walk on an undirected graph: this is the random walk which moves from $i$ to $j$ with probability 
$1/(4\max(d(i),d(j)))$ whenever $i$ and $j$ are neighbors. Moreover, given a random walk on a graph, the hitting time from $i$ to $j$ is defined to be the expected time until
the walk visits $j$ starting from $i$. We will use $d_{\rm max}$ to refer to the largest degree of any node in the sequence $G(t)$ and $M$ to refer to the
largest number of nodes that have a measurement at any one time; clearly both $d_{\rm max}$ and $M$ are at most $n$. Finally,  $\lceil x \rceil$ denotes the smallest integer which is at least $x$, and recall that $l$ is the dimension of $\mu$ and all $v_i(t)$. With this notation in place, we now state our main result. 

\bigskip

\begin{theorem} \label{mainthm} Let the stepsize be $\Delta(t)=1/t^{1-\epsilon}$ for some $\epsilon \in (0,1)$. Suppose each of the graphs  $G(t)$ is connected and let $\mathcal{H}$ be the largest hitting time from any node to any node in a very lazy
Metropolis walk on any of the graphs $G(t)$. If $t$ satisfies the lower bound
\begin{equation} \label{transient} t \geq 2T \left[ \frac{288T  \mathcal{H} }{\epsilon} \ln \left( \frac{96T   \mathcal{H} }{\epsilon}  \right) \right]^{1/\epsilon}, \end{equation}
 then we have the following decay bound on the expected variance: 
\begin{equation} \label{eq:connected} 
 E[ Z(t) ~|~ v(1) ]  \leq   39 {\cal H} T l \frac{M \sigma^2 + n T(\sigma')^2}{(t/T-1)^{1 -  \epsilon}} \ln t + Z(1) e^{-\frac{(t/T-1)^\epsilon - 2}{24 \mathcal{H} T\epsilon}}.  \end{equation}

In the general case when each $G(t)$ is not necessarily connected but the sequence $G(t)$ is $B$-connected, we have that if $t$ satisfies the lower bound
\[ t \geq 2 \max(T,2B) \left[ \frac{384 n^2 d_{\rm max} \left( 1 + \max(T,2B) \right) }{\epsilon} \ln \left( \frac{128 n^2 d_{\rm max} \left( 1 + \max(T,2B) \right) }{\epsilon} \right) \right]^{1/\epsilon} \] then 
we have the following decay bound on the expected variance: \begin{equation} \label{eq:general} E[ Z(t) ~|~ v(1) ] \leq 51 n^{2} d_{\rm max} ( 1 +  \max(T,2B))^2 l \frac{ M \sigma^2 + n  (\sigma')^2}{\left( t/\max(t,2B) \right)^{1 - \epsilon}} \ln t + Z(1) e^{- \frac{\left(t/\max(T,2B) \right)^\epsilon - 2}{32n^2 d_{\rm max}  \left( 1 + \max(T,2B) \right) \epsilon}}. \end{equation}
\end{theorem} 

Our theorem provides a quantitative bound on the convergence time of the repeated alignment process of Eq. (\ref{nonmeasuringupdate}) and Eq. (\ref{measuringupdate}). We believe this is the first time a convergence time result has been demonstrated in the setting of time-varying (not necessarily connected) graphs, 
intermittent measurements by possibly different nodes, and noisy communications among nodes. The convergence
time expressions are somewhat unwieldy, and we pause now to discuss some of their features. 

First, observe that the convergence times are a sum of two terms: the first which decays with $t$ as $O((\ln t)/t^{1-\epsilon})$ and the second which decays as 
$O(e^{-c \cdot t^{\epsilon}})$ for some parameter $c$ (here $O$-notation hides all terms that do not depend on $t$). In the limit of large $t$, the second will be negligible and we may focus 
our attention solely on the first. Thus our finding is that it is possible to achieve a nearly linear decay with time by picking a stepsize $1/t^{1-\epsilon}$ with
$\epsilon$ close to zero.

Moreover, examining Eq. (\ref{eq:general}), we find that for every choice of $\epsilon \in (0,1)$, the scaling with the number of nodes $n$ is polynomial.  Moreover, in analogy to some recent work on consensus \cite{noot09}, better convergence time bounds are available when the largest degree of any node is small. This is somewhat counter-intuitive  since higher degrees are associated with improved connectivity. A plausible intuitive explanation for this mathematical 
phenomenon is that low degrees ensure that the influence of new measurements on nodes does not get repeatedly diluted in the update process.

Furthermore, while it is possible to obtain a nearly linear decay with the number of iterations $t$ as we just noted, such a 
choice blows up the bound on the transient period before the asymptotic decay bound kicks in. Every choice of $\epsilon$ then provides a trade off between the transient size and the asymptotic rate of decay. This is to be contrasted with the usual situation in distributed optimization (see e.g., \cite{ram-angelia, sn11}) where a specific choice of stepsize usually
results in the best bounds.  


Finally, in the case when all graphs are connected, the effect of network topology on the convergence time comes through the maximum hitting time $\mathcal{H}$ in all the 
individual graphs $G(t)$. There are a variety of results on hitting times for various graphs which may be plugged into Theorem \ref{mainthm} to obtain precise
topology-dependent estimates. We first mention the general result that $\mathcal{H} = O(n^2)$ for the Metropolis chain on an arbitrary connected graph from 
\cite{metropolis}. On a variety of reasonably connected graphs, hitting times are considerably smaller. A recent preprint \cite{luxburg} shows that for many graphs, hitting times are proportional to the inverse degrees. In a 2D or 3D grid, we have that ${\cal H} = \widetilde{O}(n)$ \cite{covertime}, where the notation $\widetilde{O}( f(n) )$ is the same as ordinary $O$-notation with the exception of hiding multiplicative factors which are polynomials in $\log n$. 

We illustrate the convergence times of Theorem \ref{mainthm} with a concrete example. Suppose we have a collection of nodes interconnected in 
(possibly time-varying) 2D grids with a single (possibly different) node sampling at every time. We are interested how the time until $E[Z(t) ~|~ v(1)]$ falls below
$\delta$ scales with the number of nodes $n$ as well as with $\delta$. Let us assume that the dimension $l$ of the vector we are learning as well as the noise variance $\sigma^2$ are constants independent of the number of nodes. Choosing a step size $\Delta(t) = 1/\sqrt{t}$, we have that Theorem \ref{mainthm} implies that
variance $E[Z(t) ~|~ Z(0)]$ will fall below $\delta$ after $\widetilde{O}(n^2/\delta^2)$ steps of the protocol. The exact bound, with all the constants, 
may be obtained from Eq. (\ref{eq:connected}) by plugging in the hitting time of the 2D grid \cite{covertime}.  Moreover, the transient period until this exact bound applies (from Eq. (\ref{transient})) has length $\widetilde{O}(n^2)$. We can obtain a better asymptotic decay by picking a more slowly decaying stepsize, at the expense of lenghtening the transient period. 



\subsection{Related work} We believe that our paper is the first to \ao{derive rigorous convergence time results for} the problem of cooperative multi-agent learning by a network \ao{subject to} unpredictable communication disruptions and intermittent measurements. The key features of our model are 1) its cooperative nature (many nodes working together) 2) its reliance only on distributed and local observations 3) the incorporation of time-varying communication restrictions 4) noisy measurements and noisy communication. 

Naturally, our work is not the first attempt to fuse learning algorithms with distributed control or multi-agent settings. \ao{Indeed, the study of learning in games is a classic subject which has attracted considerable attention within the last couple of decades due in part to its applications to multi-agent systems.} We refer the reader to the recent papers \cite{ma1, ma2, ma3, ma4, ma5, ma6, ma7, ma8, ma9, ma10, jadb-learn, jadb-more-learn} \ao{as well as the classic works \cite{ma11, ma13}} which study multi-agent learning in a game-theoretic context.  \ao{Moreover, the related problem of distributed reinforcement learning has attracted some recent attention; we refer the reader to \cite{ma11, ma12, ma14}.} We mention especially the recent surveys \ao{\cite{ma15, ma16}}. \ao{Moreover, we note that much of the recent literature in distributed robotics has focused on distributed algorithms robust to faults and communication link failures. We refer the reader to the representative papers \cite{r1, r2}. }

Our work here is very much in the spirit of the recent literature on distributed filtering \cite{olfati1, olfati2, olfati3, rantzer, sper, sayed1, sayed2, sayed3, moura1, moura6, diffadpt} and
especially \cite{carli}. These works consider the problem of tracking a time-varying signal from local measurements by each node, which are then repeatedly combined through
a consensus-like iteration. The above-referenced papers consider a variety of schemes to this effect and obtain bounds on their performance, usually stated in 
terms of solutions to certain Lyapunov equations, or in terms of eigenvalues of certain matrices on fixed graphs. 

Our work is most closely related to a number of recent papers on distributed detection \cite{moura1, moura2, moura3, moura5, moura6, moura7, moura8, moura9} which seek to evaluate protocols for networked cooperative hypothesis testing and related problems. Like the previously mentioned work on distributed filtering, these papers use the idea of local iterations which are combined through a distributed consensus update, termed ``consensus plus innovations''; a similar idea is called ``diffusion adaptation'' in \cite{diffadpt}. This literature clarified a number of distinct phenomena in cooperative filtering and estimation; some of the contributions include working out tight bounds on error exponents for choosing the right hypothesis and other performance measures for a variety of settings (e.g., \cite{moura2, moura3, moura5}), as well as establishing a number of fundamental limits for distributed parameter estimation \cite{moura7}. 

In this work, we consider the related (and often simpler) question of learning a static unknown vector. However, 
we derive results which are considerably stronger compared to what is available in the previous literature, obtaining convergence rates in settings when the network is time-varying, measurements are intermittent, and communication is noisy. {\em Most importantly, we are able to explicitly bound the speed of convergence to the unknown vector $\mu$ in these unpredictable settings in terms of network size and the combinatorial features (i.e. hitting times) of the networks.}

\subsection{Outline} We now outline the remainder of the paper. Section \ref{sec:conv}, which comprises most of our paper, contains the proof of Proposition \ref{thm:conv} as well as the main result, Theorem \ref{mainthm}. The proof is broken up into several distinct pieces since some steps 
are essentially lengthy exercises in analysis. We begin in Section \ref{subsec:aclass} which contains some basic facts about symmetric substochastic
matrices which will be useful. The following Section \ref{subsec:decay} is devoted solely to analyzing a particular inequality. We will later show 
that the expected variance satisfies this inequality and apply the decay bounds we derived in that section.  We 
then begin analyzing properties of our protocol in Section \ref{subsec:sieve}, before finally proving Proposition \ref{thm:conv} and Theorem \ref{mainthm} in Section \ref{sec:proof}.  Finally, Section \ref{sec:simul} contains some simulations of our protocol and
Section \ref{sec:concl} concludes with a summary of our results and a list of several open problems. 
 
%
 
\section{Proof of the main result\label{sec:conv}}

The purpose of this section is to prove Theorem \ref{mainthm}; we prove Proposition \ref{thm:conv} along the way.  
We note that the first several subsections contain some basic results which
we will have occasion to use later; it is only in Section \ref{sec:proof} that we begin directly proving Theorem \ref{mainthm}.
We begin with some preliminary definitions. 


\subsection{Definitions}
Given a nonnegative matrix $A \in \R^{n \times n}$, we will use $G(A)$ to
denote the graph whose edges correspond to the positive entries of $A$ in the following way: $G(A)$ is the directed graph on the vertices
$\{1,2, \ldots, n\}$ with edge set $\{ (i,j) ~|~ a_{ji}>0 \}$.  Note that if $A$ is symmetric then the graph $G(A)$ will be \ao{undirected}. We will use the standard convention of $\e_i$ to mean
the $i$'th basis column vector and $\1$ to mean the all-ones vector. Finally, we will use $r_i(A)$ to denote the row sum of the $i$'th row of $A^2$ and $R(A)={\rm diag}(r_1(A), \ldots,r_n(A))$. When the argument matrix $A$ is clear from context, we will simply write $r_i$ and $R$ for $r_i(A), R(A)$.

\subsection{A few preliminary lemmas\label{subsec:aclass}} In this subsection we prove a few lemmas which we will find useful in the proofs of our main theorem. Our first lemma gives a decomposition of a symmetric matrix and its immediate corollary provides a way to bound the change in norm arising from multiplication by a symmetric matrix. Similar statements were proved in \cite{bdx04},\cite{noot09}, and \cite{TN11}. 

\smallskip

\begin{lemma} For any symmetric matrix $A$,  \[ A^2 = \ao{R} - \sum_{k<l} [A^2]_{kl} (\e_k - \e_l) (\e_k - \e_l)^T.\] \label{lemma:decomposition} \end{lemma}

\begin{proof} Observe that each term $(\e_k - \e_l) (\e_k - \e_l)^T$ in the sum on the right-hand side has row sums of zero, and consequently both sides of the above equation have identical row sums. Moreover, both sides of the above equation are symmetric. This implies it suffices to prove that all the $(i,j)$-entries of both sides
with $i<j$ are the same. But on both sides, the $(i,j)$'th element when $i<j$ is $[A^2]_{ij}$.  \end{proof}

\smallskip

\ao{This lemma may be used to bound how much the norm of a vector changes after multiplication by a symmetric matrix.}

\smallskip

\begin{corollary} \label{sievebound} For any symmetric matrix $A$,
\[ ||Ax||_2^2  ~=~  ||x||_2^2 - \sum_{j=1}^n (1-r_j) x_j^2 \aor{-} \sum_{k<l} [A^2]_{kl}(x_k - x_l)^2.  \]    \end{corollary}

\begin{proof} By Lemma \ref{lemma:decomposition}, \begin{eqnarray*} ||Ax||_2^2 & = & 
 x^T A^2 x \\
& = & x^T R x - \sum_{k<l} [A^2]_{kl} x^T (\e_k - \e_l) (\e_k - \e_k)^T x \\
& = & \sum_{j=1}^n r_j x_j^2 - \sum_{k<l} [A^2]_{kl} (x_k - x_l)^2. 
\end{eqnarray*} Thus the decrease in squared norm from $x$ to $Ax$ is 
\[ ||x||_2^2 - ||Ax||_2^2 = \sum_{j=1}^n (1-r_j) x_j^2 + \sum_{k<l} [A^2]_{kl}(x_k - x_l)^2. \] \end{proof}

\bigskip We now introduce a measure of graph connectivity which we call the {\em sieve constant} of a graph, defined as follows. For a nonnegative, stochastic matrix $A \in \R^{n \times n}$, the sieve constant $\kappa(A)$ is defined as \[ \kappa(A) = \min_{m=1, \ldots, n} ~~\min_{||x||_2=1} ~~ x_m^2 + \sum_{k \neq l} a_{kl} (x_k - x_l)^2. \]  For an undirected graph $G=(V,E)$,  the sieve constant $\kappa(G)$ denotes the sieve constant of the Metropolis matrix, which is the stochastic matrix with
 \begin{equation*}
a_{ij} =
\begin{cases} \frac{1}{2 \max(d_i, d_j)}, & \text{ if } (i,j) \in E \mbox{ and } i \neq j,\\
0, &\text{ if } (i,j) \notin E. 
\end{cases}
\end{equation*} The sieve constant is, as far as we are aware, a novel graph parameter: we are not aware of any previous works making use of it. Our name is due to the geometric picture inspired by the above optimization problem: one entry of the vector $x$ must be held close to zero while keeping it close to all the other entries, with $\kappa(A)$ measuring how much ``sieves'' through the gaps. 

\smallskip

The sieve constant will feature prominently in our proof of Theorem \ref{mainthm}; we will use it in conjunction with Lemma \ref{sievebound} to bound 
how much the norm of a vector decreases after multiplication by a substochastic, symmetric matrix. We will require a bound on how small $\kappa(A)$ can be
in terms of the combinatorial features of the graph $G(A)$; such a bound is given by the following lemma. 

\smallskip

\begin{lemma} \label{lemma:snonnegativity} For any nonnegative, stochastic $A$, we have $\kappa(A) \geq 0$. Moreover, denoting the smallest positive entry of $A$ by $\eta$, we have that if the graph $G(A)$ is weakly connected\footnote{A directed
graph is weakly connected if the undirected graph obtained by ignoring the orientations of the edges is connected.} then $$\kappa(A) \geq \frac{\eta}{n D },$$ where $D$ is the (weakly-connected) diameter of $G(A)$.  For a symmetric matrix $A$, this lower bound can be doubled. 
\end{lemma}

\smallskip

\begin{proof} It is evident from the definition of $\kappa(A)$ that it is necessarily nonnegative. Let $E$ be the set of edges obtained by considering all the (directed) edges in $G(A)$ and ignoring the orientation of each edge. We will show that  for any $m$,  
\[ \min_{||x||_2=1} ~~~~~  x_m^2 + \sum_{(i,j) \in E} (x_i - x_j)^2 \geq \frac{1}{D n}  
\] This then implies the lemma immediately  from the definition of the sieve constant. 

Indeed, we may suppose $m=1$ without loss of generality. Suppose  the minimum in the above optimization problem is achieved by the vector $x$; let $Q$ be the index of the component of $x$ with 
the largest absolute value; without loss of generality, we may suppose that the shortest path connecting $1$ and $Q$ is $1-2-\cdots -Q$ (we can simply relabel the nodes to make this true). Moreover, we may also assume $x_Q > 0$ (else, we can just replace $x$ with $-x$). 

\ao{Now the assumptions that $||x||_2=1$, that
$x_Q$ is the largest component of $x$ in absolute value, and  that $x_Q > 0$} imply that  $x_Q \geq 1/\sqrt{n}$ or
\[ (x_1-0) + (x_2 - x_1) + \cdots + (x_Q - x_{Q-1}) \geq \frac{1}{\sqrt{n}} \] and applying Cauchy-Schwarz
\[ Q ( x_1^2 + (x_2 - x_1)^2 + \cdots (x_Q - x_{Q-1})^2) \geq \frac{1}{n}, \] or 
\[ x_1^2 + (x_2 - x_1)^2 + \cdots (x_Q- x_{Q-1})^2 \geq \frac{1}{Qn} \geq \frac{1}{ D n}. \] We note this proof is inspired by a similar argument found in \cite{LO81}. \end{proof}

\subsection{A decay inequality and its consequences\label{subsec:decay}} We continue here with some preliminary results which we will use in the course of proving Theorem \ref{mainthm}. The proof of that theorem will proceed by arguing that $a(t)= E[Z(t) ~|~ Z(0)]$ will satisfy the inequality  
\begin{equation} \label{exampledecay} a(t_{k+1}) \leq  \left( 1 - \frac{q}{t_{k+1}^{1-\epsilon}} \right) a({t_k}) + \frac{d}{t_k^{2 - 2 \epsilon}} \end{equation} for some increasing integer sequence $t_k$ and some positive constants $q,d$. We will not turn to deriving this  inequality for $E[Z(t) ~|~ Z(0)]$ now; this will be done later in Section \ref{sec:proof}. The current subection is instead devoted to analyzing the consequences of the inequality, specifically deriving a bound on how fast $a(t_k)$ decays as a function of $q,d$ and the sequence $t_k$. 

The only result from this subsection which will be used later is Corollary \ref{effdecay}; all the other lemmas proved here are merely steps on the way of the 
proof of that corollary. 

We begin with a lemma which bounds some of the products we will shortly encounter.

\smallskip

\begin{lemma} Suppose $q \in (0,1]$ and $\epsilon \in (0,1)$ and for $2 \leq a < b$ define  \[ \Phi_q(a,b) = \prod_{t=a}^{b-1} \left( 1 - \frac{q}{t^{1-\epsilon}} \right). \] Moreover, we will adopt the convention that $\Phi_q(a,b)=1$ when $a=b$. Then we have 
\[ \Phi_q(a,b) \leq e^{-q(b^{\epsilon} - a^{\epsilon})/\epsilon}. \] \label{phibound}
\end{lemma}

\smallskip

\begin{proof} Taking the logarithm of the definition of $\Phi_q(a,b)$, and using the inequality $\ln(1-x) \leq -x$,
\[ \ln \Phi_q(a,b) = \sum_{t=a}^{b-1} \ln \left( 1 - \frac{q}{t^{1-\epsilon}} \right) \leq - \sum_{t=a}^{b-1} \frac{q}{t^{1-\epsilon}}
\leq  -q \frac{b^{\epsilon} - a^{\epsilon}}{\epsilon}, \] where, in the last inequality, we applied the standard technique of lower-bounding a 
nonincreasing nonnegative sum by an integral. 
\end{proof} 

\smallskip

We next turn to a lemma which proves yet another bound we will need, namely a lower bound on $t$ such that
the inequality $t \geq \beta \log t$ holds. 

\smallskip

\begin{lemma} Suppose $\beta \geq 3$ and $t \geq  3 \beta \ln \beta$. Then $\beta \ln t \leq t$. \label{betabound}
\end{lemma} 

\smallskip

\begin{proof}  On the one hand, the inequality holds at $t= 3 \beta \ln \beta$:
\[ \beta \ln (3 \beta \ln \beta) = \beta \ln 3 + \beta \ln \beta + \beta \ln \ln \beta \leq 3 \beta \ln \beta. \] On the other hand, 
the derivative of $t - \beta \ln t$ is nonnegative for $t \geq \beta$, so that the inequality continues to hold for all $t \geq 3 \beta \ln \beta$. 
\end{proof} 

\smallskip

Another useful bound is given in the following lemma. 

\smallskip

\begin{lemma} Suppose $\epsilon \in (0,1)$ and $0 < \alpha \leq b$. Then 
\[ (b-\alpha)^{\epsilon} \leq b^{\epsilon} -  \frac{\epsilon}{b^{1-\epsilon}} \alpha \] \label{epsilonpower}
\end{lemma}

\begin{proof} We may rewrite the inequality as \[ b \left( \frac{b - \alpha}{b} \right)^{\epsilon} \leq b - \epsilon a \] Note that for fixed $\alpha \leq b$ the expression on the left is a convex function of $\epsilon$ and we have equality for both $\epsilon=0$ and
$\epsilon=1$. By convexity this implies the inequality for all $\epsilon \in [0,1]$.
\end{proof} 

\smallskip

We now combine Lemma \ref{phibound} and Lemma \ref{epsilonpower} to obtain a convenient bound on $\Phi_q(a,b)$ whenever $a$ is not too close to $b$.

\smallskip

\begin{lemma} Suppose $q \in (0,1]$ and $\epsilon \in (0,1)$.  
Then if $a,b$ are integers satisfying $2 \leq a < b$ and $a \leq b - \frac{2}{q} b^{1-\epsilon} \ln (b) $,  we have \label{cubedecay}  \[ \Phi_q(a,b) \leq \frac{1}{b^2} \] 
\end{lemma}

\smallskip

\begin{proof} Indeed, observe that as a consequence of Lemma \ref{epsilonpower},  
\[ b^{\epsilon} - a^{\epsilon} \geq b^{\epsilon} - (b - \frac{2}{q} b^{1-\epsilon} \ln (b))^{\epsilon} \geq b^{\epsilon} - ( b^{\epsilon} - \frac{\epsilon}{b^{1-\epsilon}} \frac{2}{q} b^{1-\epsilon} \ln b ) =  \frac{2 \epsilon}{q}   \ln b \] and
consequently \[ e^{-q \frac{b^\epsilon - a^\epsilon}{\epsilon}} \leq e^{-2 \ln b} = \frac{1}{b^2}. \] The claim now follows by Lemma \ref{phibound}. 
\end{proof}

\smallskip

The previous lemma suggests that as long as the distance between $a$ and $b$ is at least $(2/q) b^{1-\epsilon} \ln b$, then $\Phi_q(a,b)$ will be small. 
The following lemma provides a bound on how long it takes until the distance from $b/2$ to $b$ is at least this large. 

\smallskip

\begin{lemma} Suppose $b \geq \left[ \frac{12}{q\epsilon} \ln \frac{4}{q\epsilon} \right]^{1/\epsilon}$, $q \in (0,1]$, and $\epsilon \in (0,1)$. Then $b - \frac{2}{q} b^{1-\epsilon} \ln(b) \geq b/2$. 
\label{halflemma}
\end{lemma} 
\begin{proof} Rearranging, we need to argue that 
\[ b^{\epsilon} \geq \frac{4}{q} \ln b \] Setting $t = b^{\epsilon}$ this is equivalent to 
\[ t \geq \frac{4}{q \epsilon} \ln t \] which, by Lemma \ref{betabound}, occurs if 
\[ t \geq \frac{12}{q\epsilon} \ln \frac{4}{q\epsilon} \] or 
\[ b \geq \left[ \frac{12}{q\epsilon} \ln \frac{4}{q\epsilon} \right]^{1/\epsilon}.\]
\end{proof}

\smallskip

For simplicity of presentation, we will henceforth adopt the notation \[ \alpha(q,\epsilon) =  \left[ \frac{12}{q \epsilon} \ln \left( \frac{4}{q \epsilon} \right) \right]^{1/\epsilon}. \] 

\smallskip

\smallskip With all these lemmas in place, we now turn to the main goal in this subsection, which is to analyze how a sequence satisfying
Eq. (\ref{exampledecay}) decays with time. Our next lemma does this for the special choice of $t_k=k$. The proof of this lemma relies on 
all the results derived previously in this subsection.

\smallskip

\begin{lemma} \label{decay} Suppose $a(k)$ is a nonnegative sequence satisfying \[ a({k+1}) \leq \left( 1 - \frac{q}{(k+1)^{1-\epsilon}} \right) a(k) + \frac{d}{k^{2 - 2 \epsilon}},\] where $q \in (0,1]$, $ \epsilon \in (0,1)$,  and $d$
  is nonnegative.  
Then for \[ k \geq \alpha(q,\epsilon) \] we have \[ a(k) \leq \frac{25d}{q} \frac{\ln k}{ k^{1 - \epsilon}}  + a(1) e^{-q(k^\epsilon - 2)/\epsilon}. \]
\label{decay0}
\end{lemma}

\smallskip

\begin{proof} Let us adopt the shorthand $\phi(k) = d/k^{2-2\epsilon}$. We have that 
\[ a(k) \leq \phi(k-1) + \phi(k-2) \Phi_q(k,k+1) + \phi(k-3) \Phi_q(k-1,k+1) + \cdots + \phi(1) \Phi_q(3,k+1) + a(1) \Phi_q(2,k+1). \] We will break this sum up into
four pieces:
\begin{eqnarray*} a(k) & \leq &  \sum_{j=1}^{\lfloor (2/q) k^{1-\epsilon} \ln k \rfloor} \phi(k-j) \Phi_q(k+2-j,k+1)  +  \sum_{j=\lfloor (2/q) k^{1-\epsilon} \ln k \rfloor + 1}^{\lceil (2/q) k^{1-\epsilon} \ln k \rceil + 1} \phi(k-j) \Phi_q(k+2-j,k+1) \\ && + \sum_{j= \lceil (2/q) k^{1-\epsilon} \ln k \rceil + 2}^{k-1}  \phi(k-j) \Phi_q(k+2-j,k+1)  + a(1) \Phi_q(2,k+1)   \end{eqnarray*} We will bound each piece separately. 

The first piece can be bounded by using Lemma \ref{halflemma} to argue that each of the terms $\phi(k-j)$ is at most $d(1/(k/2))^{2 - 2 \epsilon}$, the quantity $\Phi_q(k-j+2,k+1)$ is upper bounded by $1$, and there are at most $(2/q) k^{1-\epsilon} \ln k $ terms in the sum. Consequently, 
\[ \sum_{j=1}^{\lfloor (2/q) k^{1-\epsilon} \ln k \rfloor} \phi(k-j) \Phi_q(k-j+1,k)  \leq \left( \frac{2}{q} k^{1-\epsilon} \ln k   \right) \frac{d}{(k/2)^{2-2\epsilon}}  \leq \frac{8d \ln k}{q k^{1-\epsilon}}. \] In the second piece, there are at most two terms, each of the $\Phi_q(\cdot, \cdot)$ is at most one, and we use Lemma \ref{halflemma} again to argue that the piece is upper bounded by  
\[ \frac{d}{(k/2 - 1)^{2-2\epsilon}} + \frac{d}{(k/2 - 2)^{2 - 2 \epsilon}} \leq \frac{4d}{(k/2)^{2 - 2 \epsilon}} \leq \frac{16d}{k^{2-2 \epsilon}} \] where the penultimate inequality is true because $k \geq \alpha(q,\epsilon)$ implies $k \geq 12 \ln 4 \geq 16$. We bound the third piece by arguing that it is at most 
\[ \sum_{j= \lceil (2/q) k^{1-\epsilon} \ln k \rceil+2}^{k-1}  \phi(k-j) \Phi_q(k-(j-2)),k)  \] and then arguing that all the terms $\phi(k-j)$ are bounded above by $d$, whereas the sum of $\Phi_q(k-(j-2),k)$ over 
that range is at most $1/k$ due to Lemma \ref{cubedecay}. Thus
\[ \sum_{j= \lceil (2/q) k^{1-\epsilon} \ln k \rceil+2}^{k-1}  \phi(k-j) \Phi_q(k+2-j,k) \leq \frac{d}{k}. \] Finally, the last term is bounded straightforwardly by Lemma \ref{phibound}. Putting these 
three bounds together gives the statement of the current lemma. 
\end{proof}

\smallskip Finally, we turn to the main result of this subsection, which is the extension of the previous corollary to the case of general sequences 
$t_k$. The following result is the only one which we will have occasion to use later, and its proof proceeds by an appeal to Lemma \ref{decay0}.

\smallskip

\begin{corollary} Suppose $a(t_k)$ is a nonnegative sequence satisfying \[ a(t_{k+1}) \leq \left( 1 - \frac{q}{t_{k+1}^{1-\epsilon}} \right) a(t_k) + \frac{d}{t_k^{2 - 2 \epsilon}},\] where $q \in (0,1]$, $d$ is nonnegative, $\epsilon \in (0,1)$ and $t_k$ is some increasing integer sequence satisfying $t_1=1$ and 
\[ |t_{k+1} - t_k | < T ~~~~ \mbox{ for all nonnegative } k, \] where $T$ is some positive integer.  
Then if 
\[ k \geq \left[ \frac{12T}{q \epsilon} \ln \left( \frac{4T}{q \epsilon} \right) \right]^{1/\epsilon} ,\] we will have 
\[ a(t_k) \leq \frac{25dT \ln k}{q k^{1 -  \epsilon}}  + a(1) e^{-q(k^\epsilon - 2)/(T\epsilon)} .\]\label{effdecay}
\end{corollary} 

\begin{proof} Define $b(k) = a(t_k)$. Then 
\[ b(k+1) \leq \left( 1 - \frac{q}{t_{k+1}^{1-\epsilon}} \right) b(k) + \frac{d}{t_k^{2 - 2 \epsilon}}. \] Since $  k \leq t_k \leq kT$ , we have that 
\[ b(k+1) \leq \left( 1 - \frac{q/T^{1-\epsilon}}{(k+1)^{1-\epsilon}} \right) b(k) + \frac{d}{k^{2 - 2 \epsilon}}. \] Applying Lemma \ref{decay0}, we get
that for 
\[ k \geq  \left[ \frac{12T}{q \epsilon} \ln \left( \frac{4T}{q\epsilon} \right) \right]^{1/\epsilon},\] we have 
\begin{equation} \label{tkeq} b(k) \leq \frac{25dT \ln k }{qk^{1-\epsilon}} + b(1) e^{-q(k^{\epsilon}-2)/(\epsilon T)}. \end{equation} 
The corollary now follows since $a(t_k)=b(k)$. 
\end{proof}

\subsection{Analysis of the learning protocol\label{subsec:sieve}}

With all the results of the previous subsections in place, we can now turn to the analysis of our protocol. We do not begin the actual proof of Theorem \ref{mainthm} in this subsection, but rather we derive some bounds on the decrease of $Z(t)$ at each step. It is in the next Section \ref{sec:proof}, we will make use of these bounds to prove Theorem \ref{mainthm}. 

 For the remainder of Section \ref{subsec:sieve},
we will assume that $l=1$, i.e,  $\mu$ and all $v_i(t)$ belong to $\R$. We will then define $v(t)$ to be the vector that stacks up $v_1(t), \ldots, v_n(t)$. 

The following proposition describes a convenient way to write Eq (\ref{nonmeasuringupdate}). We omit the proof (which is obvious). 

\smallskip

\begin{proposition} \label{remark:eqrewrite} We can rewrite Eq. (\ref{nonmeasuringupdate}) and Eq. (\ref{measuringupdate}) as follows:
\begin{eqnarray*} y(t+1) & = &  A(t) v(t) + b(t)  \\
q(t+1) & = & (1 - \Delta(t)) v(t) + \Delta(t) y(t+1) \\
v(t+1) & = & q(t+1) + \Delta(t) r(t)\aor{ + \Delta(t) c(t)},
\end{eqnarray*} where: 
\begin{enumerate}
\item If $i \neq j$ and $i,j$ are neighbors in $G(t)$, \[ a_{ij}(t) = \frac{1}{4 \max(d_i(t), d_j(t))}. \] \ao{However, if $i \neq j$ are not neighbors in 
$G(t)$, then $a_{ij}(t)=0$.} As a consequence, 
$A(t)$ is a symmetric matrix. 
\item If node $i$ does not have a measurement of $\mu$ at time $t$, then \[ a_{ii}(t) = 1 - \frac{1}{4} \sum_{j \in N_i(t), ~~ j \neq i} 
\frac{1}{\max(d_i(t),d_j(t))}. \] On the other hand, if node $i$ does have a measurement of $\mu$ at time $t$, 
\[ a_{ii}(t) =  \frac{3}{4} - \frac{1}{4} \sum_{j \in N_i(t), ~~ j \neq i} 
\frac{1}{\max(d_i(t),d_j(t))} . \] Thus $A(t)$ is a diagonally dominant matrix and its graph is merely the 
intercommunication graph at time $t$: $G(A(t)) = G(t)$. Moreover, if no node has a measurement at time $t$, $A(t)$ is stochastic.
\item If node $i$ does not have a measurement of $\mu$ at time $t$, then $b_i(t)=0$. If node $i$ does have a measurement of $\mu$
at time $t$, then $b_i(t) = (1/4) \mu$. 
\item If node $i$ has a measurement of $\mu$ at time $t$, $r_i(t)$ is a random variable with mean zero and variance $\sigma^2/16$.  Else, 
$r_i(t)=0$. Each $r_i(t)$ is independent of all $v(t)$ \aor{and all other $r_j(t)$. Similarly, $c_i(t)$ is the random variable \[ c_i(t) = \frac{1}{4} \sum_{j \in N(i)} \frac{w_{ij}(t)}{\max( d_i(t), d_j(t) )}. \] Each $c_i(t)$ has mean zero, and the vectors $c(t), c(t')$ are independent whenever $t \neq t'$. Moreover, $c(t)$ and $r(t')$ are 
independent for all $t,t'$. Finally, $E[c_i^2(t)] \leq (\sigma')^2/16$.} 
\end{enumerate} 

Putting it all together, we may write our update as 
\[ v(t+1) = (1- \Delta(t)) v(t)  + \Delta(t) A(t) v(t) + \Delta(t) b(t) + \Delta(t) r(t) + \Delta(t) c(t). \]
\end{proposition}

Let us use $S(t)$ for the set of agents who have a measurement at time $t$. We use this notation in the   next lemma, which 
bounds the decrease in $Z(\aor{t})$ from time $t$ to $t+1$.

\smallskip

\begin{lemma} \label{lemma:measurementdecrease} If $\Delta(t) \in (0,1)$ then 
\begin{eqnarray*}
E[Z(\aor{t+1}) ~ {\bf | }~ v(t), \aor{v(t-1), \ldots, v(1)}] & \leq & Z(\aor{t}) - \frac{\Delta(t)}{\ao{8}} \sum_{(k,l) \in E(t)} \frac{(v_k(t) - v_l(t))^2}{\max (d_k(t), d_l(t))}  \\ && ~~~~~-\frac{\Delta(t)}{4} \sum_{k \in S(t)} (v_k(t) - \mu)^2 +   \frac{\Delta(t)^2}{16} \aor{ \left( M \sigma^2 + n (\sigma')^2 \right)} 
\end{eqnarray*} Recall that $E(t)$ is the set of undirected edges in the graph at time $t$, so every pair of neighbors $(i,j)$ appears once in the above sum. Moreover, if $S(t)$ is nonempty, then 
\[ E[Z(\aor{t+1}) ~ {\bf | }~ v(t), \aor{v(t-1), \ldots, v(1)}] \leq \left( 1 - \frac{1}{8} \Delta(t) \kappa \left[ G(t) \right] \right) Z(\aor{t}) +  \frac{\Delta(t)^2}{16} \aor{ \left( M\sigma^2 + n(\sigma')^2 \right)}. \]
\end{lemma}

\begin{proof} Observe that, for any $t$, the vector $\mu \1$ satisfies \[ \mu \1 = A(t) \mu \1 + b(t),\] and therefore, 
\begin{equation} \label{yminusmu} y(t+1) - \mu \1  =   A(t) (v(t) - \mu \1).
\end{equation}  We now apply Corollary \ref{sievebound} which involves the entries and row-sums of the matrix $A^2(t)$ which we lower-bound as follows. Because $A(t)$ is diagonally dominant \ao{and nonnegative}, we have that if $(k,l) \in E(t)$ then $$ [A^2(t)]_{kl} \geq [A(t)]_{kk} [A(t)]_{kl} \geq \frac{1}{2} \frac{1}{4 \max( d_k(t), d_l(t))} \geq \frac{1}{8 \max (d_k(t), d_l(t))}.$$ Moreover, if $k$ has a measurement of $\mu$ then the row
sum of the $k$'th row of $A$ \ao{equals} $3/4$, which implies that the $k$'th row sum of $A^2$ is \ao{at most} $3/4$.  Consequently,  Corollary \ref{sievebound} implies
\begin{small} \begin{equation} ||y(t+1) - \mu \1||_2^2   \leq   Z(\aor{t})  - \frac{1}{8} \sum_{(k,l) \in E(t)} \frac{(v_k(t)-v_l(t))^2}{\max (d_k(t), d_l(t))} - \frac{1}{4  } \sum_{k \in S(t)} (v_k(t) - \mu)^2.  \label{eq:subtractbound} 
\end{equation} \end{small} Next, \ao{since $\Delta(t) \in (0,1)$ we can appeal to the convexity of the squared two-norm to obtain}
\begin{small} \begin{eqnarray*}   ||q(t+1) - \mu \1||_2^2   & \leq &   Z(\aor{t})  - \frac{\Delta(t)}{8} \sum_{(k,l) \in E(t)} \frac{(v_k(t)-v_l(t))^2}{\max (d_k(t), d_l(t))} - \frac{\Delta(t)}{4} \sum_{k \in S(t)} (v_k(t) - \mu)^2. 
\end{eqnarray*} \end{small} Since $E[r(t)]= 0, \aor{E[b(t)]=0}$ and $\aor{E[||r(t) + c(t)||_2^2] \leq  (M\sigma^2 + n(\sigma')^2)/16}$ independently of \aor{all} $v(t)$, this immediately implies the first inequality in the statement of the lemma. The second inequality is then a straightforward consequence of the
definition of the sieve constant.
\end{proof} 

\subsection{Completing the proof\label{sec:proof}} With all the lemmas of the previous subections in place, we finally begin the proof of our main result,
Theorem \ref{mainthm}. Along the way, we will prove the basic convergence result of Proposition \ref{thm:conv}. 

Our first observation in this section is that it suffices to prove it in the case when $l=1$, (i.e., when $\mu$ is a number). Indeed, observe that the update 
equations Eq. (\ref{nonmeasuringupdate}) and (\ref{measuringupdate}) are separable in the entries of the vectors $v_i(t)$. Therefore, if Theorem \ref{mainthm} is proved under the assumption $l=1$, we may apply it to each component to obtain it for the general case. Thus we will therefore be assuming without loss of generality that $l=1$ for the remainder of this paper.

Our first step is to prove the basic convergence result of Proposition \ref{thm:conv}. Our proof strategy is to repeatedly apply Lemma \ref{lemma:measurementdecrease} to bound the the decrease in $Z(t)$ at each stage. This will yield a decrease rate for $Z(t)$ which will imply almost sure convergence to the correct $\mu$.

\bigskip

\begin{proof}[Proof of \aor{Proposition} \ref{thm:conv}]  We first claim that there exists some constant $c>0$ such that if $t_k=k \max(T,2B)$, then
\begin{equation} \label{constantdecay}  E[Z(\aor{t_{k+1}}) ~|~ v(t_k)] \leq (1- c \Delta(t_{k+1})) Z(\aor{t_k}) +  \max(T,2B) \Delta(t_k)^2 \aor{( M \sigma^2 + n (\sigma')^2)}. \end{equation} 
\ao{We postpone the proof of this claim for a few lines while we observe that,} as a consequence of our assumptions on $\Delta(t)$, we have the following three facts: 
\[ \sum_{k=1}^{\infty} c \Delta(t_{k+1}) = +\infty, ~~~~~~~\sum_{k=1}^{\infty}  \max(T,2B) \Delta(t_k)^2 \aor{( M \sigma^2 + n (\sigma')^2)} < \infty \] \[ \lim_{\ao{k} \rightarrow \infty} \frac{ \max(T,2B) \Delta(t_k)^2 \aor{( M \sigma^2 + n (\sigma')^2)}}{c \Delta(t_{k+1})} = 0. \] Moreover, it is true for large enough $k$ that $c \Delta( t_{k+1} ) < 1$. \ao{Now as a consequence of these four facts, Lemma 10 from Chapter 2.2 of \cite{p87}} implies $\lim_{t \rightarrow \infty} Z(\aor{t})=0$ with probability $1$. 

\ao{To conclude the proof, it remains to demonstrate} Eq. (\ref{constantdecay}). \aor{Applying Lemma \ref{lemma:measurementdecrease} at every time $t$ between $t_{k}$ and $t_{k+1}-1$ and iterating expectations, we obtain} \begin{small}
\begin{eqnarray} E[Z(\aor{t_{k+1}}) ~|~ v(t_k)] & \leq &  Z(\aor{t_k}) -   \sum_{m=t_k}^{t_{k+1}-1} \left(  \frac{\Delta(m)}{8} \sum_{(k,l) \in E(m)} \frac{E[(v_k(m)- v_l(m))^2 ~|~ v(t_k)]}{\max (d_k(m), d_l(m))} \nonumber \right. \\ && \left. ~~~~~~~~+ \frac{\Delta(m)}{4  }  \sum_{k \in S(m)} E [(v_k(m)  - \mu)^2 ~|~ v(t_k)]  \right)
 + \Delta^2(t_k) \max(T,2B) \frac{\aor{ M \sigma^2 + n (\sigma')^2}}{16}   \label{decreaseineq}
 \end{eqnarray} \end{small} Note that if $Z(t_k)=0$, then the last equation immediately implies Eq. (\ref{constantdecay}). If $Z(t_k) \neq 0$, then observe that Eq. (\ref{constantdecay}) would follow  from the assertion
\[ \inf ~~ \frac{\sum_{m=t_k}^{t_{k+1}-1} \sum_{(k,l) \in E(m)} E[(v_k(m) - v_l(m))^2 ~|~ v(t_k)] + \sum_{k \in S(m)} E[(v_k(m) - \mu)^2 ~|~ v(t_k)] }{\sum_{i=1}^n (v_i(t_k) - \mu)^2} > 0 \] where the infimum is taken over all vectors $v(t_k)$ such that $v(t_k) \neq \mu \1$ and over all possible sequences 
of undirected communication graphs and measurements between time $t_k$ and $t_{k+1}-1$ satisfying the conditions of uniform connectivity and uniform measurement speed. Now since $E[X^2] \geq E[X]^2$, we have that \begin{eqnarray*} && \inf ~~ \frac{\sum_{m=t_k}^{t_{k+1}-1} \sum_{(k,l) \in E(m)} E[(v_k(m) - v_l(m))^2 ~|~ v(t_k)] + \sum_{k \in S(m)} E[(v_k(m) - \mu)^2 ~|~ v(t_k)] }{\sum_{i=1}^n (v_i(t_k) - \mu)^2} \\ && \geq  \inf ~~ \frac{\sum_{m=t_k}^{t_{k+1}-1} \sum_{(k,l) \in E(m)} E[v_k(m) - v_l(m) ~|~ v(t_k)]^2 + \sum_{k \in S(m)} E[v_k(m) - \mu ~|~ v(t_k)]^2 }{\sum_{i=1}^n (v_i(t_k) - \mu)^2}.  \end{eqnarray*} We will \ao{complete the proof by arguing} that this last infimum is strictly positive.

Let us \ao{define} $z(t) = E[v(t) - \mu \1 ~|~ v(t_k)]$ for $t \geq t_k$.  \ao{From Proposition \ref{remark:eqrewrite} and Eq. (\ref{yminusmu}), we can work out 
the dynamics satisfied by the sequence $z(t)$ for $t \geq t_k$:  \begin{eqnarray} z(t+1) & = &  E[v(t+1) - \mu \1 ~|~ v(t_k)] \nonumber \\
& = & E[q(t+1) - \mu \1 ~|~ v(t_k)] \nonumber \\
&  = & E[(1-\Delta(t))v(t) + \Delta(t) y(t+1) - \mu \1 ~|~ v(t_k)] \nonumber \\
& = & E[ (1-\Delta(t))(v(t) - \mu \1) ~|~ v(t_k)] + E[\Delta(t) (y(t+1) - \mu \1) ~|~ v(t_k) ] \nonumber  \\
& = & E[ (1-\Delta(t))(v(t) - \mu \1) ~|~ v(t_k)] + E[\Delta(t) A(t) (v(t) - \mu \1) ~|~ v(t_k) ] \nonumber  \\
& = & \left[ (1- \Delta(t)) I + \Delta(t) A(t) \right] z(t) \label{zequation}.\end{eqnarray}}

Clearly, we need to argue that \begin{equation} \label{zinf} \inf ~~ \frac{\sum_{m=t_k}^{t_{k+1}-1} \sum_{(k,l) \in E(m)} (z_k(m) - z_l(m))^2 + \sum_{k \in S(m)} z_k^2(m) }{\sum_{i=1}^n z_i^2(t_k)} > 0  \end{equation} where the infimum is taken over all sequences of undirected communication graphs satisfying the conditions of uniform connectivity and measurement speed and \ao{all nonzero} $z(t_k)$ (which in turn determines all the $z(t)$ with $t \geq t_k$ through Eq. (\ref{zequation})). 

From Eq. (\ref{zequation}), we have that the expression within the infimum in Eq. (\ref{zinf}) is invariant under scaling of $z(t_k)$. So
we can conclude that, for any sequence of graphs $G(t)$ and measuring sets $S(t)$,  the infimum is always achieved by some vector $z(t_k)$ with $||z(t_k)||_2=1$. 


Now given a sequence of graphs $G(t)$ and a sequence of measuring sets $S(t)$, suppose $z(t_k)$ is a vector with unit norm that achieves this infimum.  Let $S_+ \subset \{1, \ldots, n\}$ be the set of indices $i$ with $z_i(t_k) > 0$, $S_-$ be the set of indices $i$ with $z_i(t_k) < 0$, and $S_0$ be the set of indices with $z_i(t_k)=0$. Since $||z(t_k) - \mu \1||_2 = 1$ we have that at least one of $S_+, S_-$ is nonempty. Without loss of generality, suppose that $S_-$ is nonempty. Due to the conditions of uniform connectivity and uniform measurement speed, there is a first time $t_k \leq t' < t_{k+1}$ when at least one of the following two events happens: (i) some node $i' \in S_-$ is connected to a node $j' \in S_0 \cup S_+$ (ii) some node $i' \in S_-$ has a measurement of $\mu$. 

In the former case, $z_{i'}(t')<0$ and $z_{j'}(t') \geq 0$ and consequently $(z_{i'}(t')-z_{j'}(t'))^2$ will be positive; in the latter case, $z_{i'}(t')<0$ and consequently
$z_{i'}^2(t')$ will be positive. 

We have thus shown that for every sequence of graphs $G(t)$ and sequence of measuring sets $S(t)$ satisfying our assumption and all $z(t_k)$ with $||z(t_k)||_2 = 1$ we have that the expression under the infimum in Eq. (\ref{zinf}) is strictly positive. Note that since the expression under the infimum is a continuous function of $z(t_k)$, this implies that, in fact, the infimum over all $z(t_k)$ with $||z(t_k)||_2=1$ is strictly positive. Finally, since there are only finitely many sequences of graphs $G(t)$ and measuring sets $S(t)$ of length $\max(T,2B)$, this proves Eq. (\ref{zinf}) and concludes the proof. 
\end{proof}

\smallskip Having established Proposition \ref{thm:conv}, we now turn to the proof of Theorem \ref{mainthm}. We will split the proof into several chunks. Recall
that Theorem \ref{mainthm} has two bounds: Eq. (\ref{eq:connected}) which holds when each graph $G(t)$ is connected and Eq. (\ref{eq:general}) which holds
in the more general case when the graph sequence $G(t)$ is $B$-connected. We begin by analyzing the first case. Our first lemma towards that end
 provides an upper bound on the eigenvalues of the matrices $A(t)$ corresponding to connected $G(t)$. 

\smallskip

\begin{lemma} Suppose each $G(t)$ is connected and at least one node makes a measurement at every time $t$. Then the largest eigenvalues of the matrices $A(t)$ satisfy  $$\lambda_{\rm max}(A(t)) \leq 1 - \frac{1}{24 \mathcal{H}} ~~ \mbox{ for all } t,$$ \label{conneig}
\noindent where, recall, $\mathcal{H}$ is an upper bound on the hitting times in the very lazy Metropolis walk on $G(t)$. 
\end{lemma}

\smallskip

\begin{proof} Let us drop the argument $t$ and simply refer to $A(t)$ as $A$. Consider the iteration 
\begin{equation} p(k+1)^T = p(k)^T A. \label{aupdate} \end{equation} We argue that it has a probabilistic interpretation. Namely, let us transform the matrix $A$ into a stochastic matrix $A'$ in the following way: we introduce a new node $i'$ for every row $i$ of $A$ which has row sum less than $1$ and set 
\[ [A']_{i,i'} = 1 - \sum_{j} [A]_{ij}, ~~~ [A']_{i',i'} = 1. \] Then $A'$ is a stochastic matrix, and moreover, observe that by construction $[A']_{i,i'}=1/4$ for every new node $i'$ that is introduced. Let us adopt the notation $\mathcal{I}$ to be the set of new nodes $i'$ added in this way, and we will use $\mathcal{N} = \{1, \ldots, n\}$ to refer to the original nodes.

We then  have that if $p(0)$ is a stochastic vector (meaning it has nonnegative entries which sum to one), then $p(k)$ generated by Eq. (\ref{aupdate}) has the following interpretation: $p_j(k)$ is the probability that a random walk within initial distribution $p(0)$, taking steps according to $A'$, is at node $j$ at time $k$.  This is easily seen by induction:
clearly it holds at time $0$, and if it holds at time $k$, then it holds at time $k+1$ since $[A]_{ij} = [A']_{ij}$ if $i,j \in \mathcal{N}$ and $[A']_{i',j}=0$ for all $i' \in \mathcal{I}$ and $j \in \mathcal{N}$. 

Next, we note  that  $\mathcal{I}$ is an absorbing set for the Markov chain with transition matrix $A'$, and moreover $||p(k)||_1$ is the probability that the
random walk starting at $p(0)$ is not absorbed in $\mathcal{I}$ by time $k$. Defining $T_i'$ to be the expected time until a random walk with transition matrix $A'$ starting from $i$ is absorbed in the set $\mathcal{I}$, we have that by Markov's inequality, if $p(0)={\bf e}_i$ then for any $t \geq 2 \lceil T_i' \rceil$,
\[ \left| \left|p (t) \right| \right|_1  \leq  \frac{1}{2} \] Therefore, for any stochastic $p(0)$, 
\[ \left| \left|p \left( 2 \lceil \max_{i \in \mathcal{N}} T_i' \rceil \right) \right| \right|_1  \leq  \frac{1}{2}. \] Thus for any nonnegative integer $k$, 
\[ \left| \left|p \left( 2 k \lceil \max_{i \in \mathcal{N}} T_i' \rceil \right) \right| \right|_1 \leq \left( \frac{1}{2} \right)^k. \] This implies that for all initial vectors $p(0)$ (not necessarily stochastic ones), 
\[ \left| \left| p \left( 2 k \lceil \max_{i \in \mathcal{N}} T_i' \rceil \right) \right| \right|_1 \leq \left( \frac{1}{2} \right)^k ||p(0)||_1 .\] By the Perron-Frobenius theorem, $\lambda_{\rm max}(A)$ is real and its corresponding eigenvector is real. Plugging it in for $p(0)$, we get that 
\[ \lambda_{\rm max}^{2k \lceil \max_{i \in \mathcal{N}} T_i' \rceil} \leq \left( \frac{1}{2} \right)^k \] 
or \begin{equation} \label{lambdap}  \lambda_{\rm max} \leq \left( \frac{1}{2} \right)^{1/( 2 \lceil \max_{i \in \mathcal{N}} T_i' \rceil)} \leq 1 - \frac{1}{4 \lceil \max_{i \in \mathcal{N}} T_i' \rceil} \end{equation} where we used the inequality $(1/2)^x \leq 1 - x/2$ for all $x \in [0,1]$.

 It remains to rephrase this bound in terms of the hitting times in 
the very lazy Metropolis walk on $G$. We simply note that $[A']_{i,i'} = 1/4$ by construction, so 
\[ \max_i T_i' \leq 4 ( \mathcal{H} + 1 ). \] Thus 
\[ \lceil \max_i T_i' \rceil \leq 4 (\mathcal{H}+1) + 1 \leq 6 \mathcal{H}, \] where we used the fact that ${\cal H} \geq 4$. Combining this bound with Eq. (\ref{lambdap}) proves the lemma. 
\end{proof} 

\smallskip

\ao{With this lemma in place, we can now proceed to prove the first half of Theorem \ref{mainthm}, namely the bound of  Eq. (\ref{eq:connected}). }

\bigskip

\begin{proof}[Proof of Eq. (\ref{eq:connected})]  We use Proposition \ref{remark:eqrewrite} to write a bound for $E[Z(t+1) ~|~ v(t)]$ in terms of the largest eigenvalue $\lambda_{\rm max}(A(t))$. Indeed, as we remarked previously in Eq. (\ref{yminusmu}), 
\[ y(t+1) - \mu {\bf 1} = A(t) (v(t) - \mu {\bf 1} ) \] we therefore have that 
\[  ||y(t+1) - \mu {\bf 1} ||_2^2 \leq \lambda_{\rm max}^2(A(t)) Z(t) \leq \lambda_{\rm max}(A(t)) Z(t) \] where we used the fact that $\lambda_{\rm max}(A(t)) \leq 1$ since $A(t)$ is a substochastic matrix. Next, we have  
\[ ||q(t+1) - \mu {\bf 1}||_2^2 \leq \left[ 1 - \Delta(t) (1 - \lambda_{\rm max}(A(t)) \right] Z(t) \] and finally
\[ E [ Z(t+1) ~|~ v(t) ] \leq \left[ 1 - \Delta(t) \left( 1 - \lambda_{\rm max}(A(t)) \right) \right] Z(t) + \Delta(t)^2 \frac{\aor{ M\sigma^2 + n(\sigma')^2}}{16}. \]
Let $t_k$ be the times when a node has had a measurement.  Applying the above equation at times $t_k, t_{k}+1, \dots, t_{k+1}-1$ and using the eigenvalue bound of Lemma \ref{conneig} at time $t_{k}$ and the trivial bound $\lambda_{\rm max}(A(t)) \leq 1$ at times $t_{k}+1, \ldots, t_{k+1}-1$, we obtain, 
\begin{eqnarray*} E[Z(t_{k+1}) ~|~ v(t_k)] & \leq & \left( 1 - \frac{1}{24 \mathcal{H} t_k^{1-\epsilon}} \right) Z(t_k) + \frac{ M \sigma^2 + nT (\sigma')^2 }{16 t_k^{2 - 2 \epsilon}} \\ 
& \leq & \left( 1 - \frac{1}{24 \mathcal{H}  t_{k+1}^{1-\epsilon}} \right) Z(t_k) + \frac{M \sigma^2 + nT(\sigma')^2}{16 t_k^{2 - 2 \epsilon}}.
\end{eqnarray*} Iterating expectations and applying Corollary \ref{effdecay}, we obtain that for
\begin{equation} \label{kbound} k \geq \left[ \frac{288 T  \mathcal{H} }{\epsilon} \ln \left( \frac{96 T \mathcal{H} }{\epsilon} \right) \right]^{1/\epsilon}, 
\end{equation} we have
\[ E [ Z(t_k) ~|~ v(1) ] \leq \frac{25 (1/16)  (M\sigma^2 + nT(\sigma')^2)24 {\cal H} T \ln k}{k^{1 - \epsilon}}  + Z(1) e^{-(k^\epsilon - 2)/(24 \mathcal{H} T\epsilon)}. \]  Using the inequality $k \leq t_k \leq kT$, 
\begin{equation} \label{expect-tk} E [ Z(t_k) ~|~ v(1) ] \leq 38 {\cal H} T  \frac{   M \sigma^2 + n T(\sigma')^2}{(t_k/T)^{1 - \epsilon}} \ln (t_k) + Z(1) e^{-((t_k/T)^\epsilon - 2)/(24 \mathcal{H} T\epsilon)}.
\end{equation} 
Finally, for any $t$ satisfying  \begin{equation} \label{tlower} t \geq T +  T \left[ \frac{288 T  \mathcal{H} }{\epsilon} \ln \left( \frac{96 T \mathcal{H} }{\epsilon} \right) \right]^{1/\epsilon} \end{equation} there is some $t_k$ with $k$ satisfying Eq. (\ref{kbound}) within the last $T$ steps before $t$. We can therefore get an upper bound
on $E[Z(t) ~|~ Z(1)]$ applying Eq. (\ref{expect-tk}) to that last $t_k$, and noting that the expected increase from that $Z(t_k)$ to $Z(t)$ is bounded as 
\[ E[Z(t) ~|~ v(1)] - E[Z(t_k) ~|~ v(1)] \leq \frac{nT (\sigma')^2}{t_k^{2-2\epsilon}} \leq \frac{n T (\sigma')^2}{t_k^{1-\epsilon}}  \leq \frac{n T (\sigma')^2}{(t/T-1)^{1-\epsilon}} \]  This implies that for $t$ satisfying Eq. (\ref{tlower}), we have
\[ E[ Z(t) ~|~ v(1) ]  \leq  39 {\cal H} T \ln t\frac{M \sigma^2 + n T(\sigma')^2}{(t/T-1)^{1 - \epsilon}} \ln t + Z(1) e^{-((t/T-1)^\epsilon - 2)/(24 \mathcal{H} T\epsilon)}. \]  After some simple algebra, this is exactly the bound of the theorem. 
 \end{proof}
 
\bigskip

 We now turn to the proof of the second part of Theorem \ref{mainthm}, namely Eq. (\ref{eq:general}). Its proof requires a certain inequality between quadratic forms in the vector $v(t)$ which we 
 separate into the following lemma. 
 
 \bigskip
 
 \begin{lemma} Let $t_1 = 1$ and $t_k = (k-1) \max(T,2B)$ for $k > 1$, and assume that the entries of the vector $v(t_k)$ satisfy  \[ v_1(t_k) < v_2(t_k) < \cdots < v_n(t_k). \]  Further, let us assume that none of the $v_i(t_k)$ equal $\mu$, and let us define $p_-$ to be the largest index such that $v_{p-}(t_k)<\mu$ and $p_+$ to be the smallest index such that
 $v_p(t_k)> \mu$.  We then have \begin{eqnarray} \sum_{m=t_k}^{t_{k+1}-1} \sum_{(k,l) \in E(m)} E[v_k(m)-v_l(m) ~|~ v(t_k)]^2 +  \sum_{k \in S(m)} E[v_k(m) - \mu ~|~ v(t_k)]^2  & \geq & \nonumber\\  (v_{p_-}(t_k) - \mu)^2 + (v_{p+}(t_k) - \mu)^2 +  \sum_{i=1, \ldots, n, ~~ i \neq p_{-}}  (v_i(t_k) - v_{i+1}(t_k))^2   &&  \label{eq:quantdecbound}  \end{eqnarray} \label{lemma:quantdecbound} 
 \end{lemma} 
 
 \begin{proof} \ao{\ao{The proof parallels a portion of the proof of {\aor Proposition} \ref{thm:conv}. First, we change variables by defining $z(t)$ as \[ z(t) = E[ \aor{v}(t) - \mu \1 ~|~ v(t_k)] \] for $t \geq t_k$.  We claim that  
\begin{equation} \label{claimeq} \sum_{m=t_k}^{t_{k+1}-1} \sum_{(k,l) \in E(m)} (z_k(m)-z_l(m))^2 +  \sum_{k \in S(m)} z_k^2(m)  \geq   z_{p_-}^2(t_k) + z_{p_+}^2(t_k) +  \sum_{i=1, \ldots, n-1, ~~ i \neq p_-} (z_i(t_k) - z_{i+1}(t_k))^2  \end{equation} The claim immediately implies the current lemma. } }
 
\ao{Now we turn to the proof of the claim, \ao{which is similar to the proof of a lemma from \cite{noot09}.} We will associate to each term on the right-hand side of
Eq. (\ref{claimeq}) a term on the left-hand side of Eq. (\ref{claimeq}), and we will argue that each term on the left-hand side is at least is big as the sum of
all terms on the right-hand side associated with it.  }

\ao{To describe this association, we first introduce some new notation. We denote the set of nodes $\{1, \ldots, l\}$ by $S_l$; its complement, the set $\{l+1, \ldots, n\}$, is then $S_l^c$.  If $l \neq p_-$, we will abuse notation by saying that $S_l$ contains zero if $l \geq p_+$; else, we say that $S_l$ does not contain zero and $S_l^c$ contains zero. However, in the case of $l=p_-$, we will say that neither $S_{p_-}$ nor $S_{p_-}^c$ contains zero. We will say that $S_l$ ``is crossed by an edge''  at time $m$ if a node in $S_l$ is connected to a node in $S_l^c$ at time $m$. For $l \neq p_-$, we will say that $S_l$ is ``crossed by a measurement'' at time $m$ if a node in whichever of $S_l,S_l^c$ that does not contain zero has a measurement at time $m$. We will say that $S_{p_-}$ is ``crossed by a measurement from the left'' at time $m$ if a node in $S_{p_-}$  has a measurement at time $m$; we will say that it is ``crossed by a measurement from the right'' at time $m$ if a node in $S_{p_-}^c$ had a measurement at time $m$. Note that the assumption of uniform connectivity means that every $S_l$ is crossed by an edge at least one time $m \in \{ t_k, \ldots, t_{k+1}-1\}$. It may happen that $S_l$ is also crossed by measurements, but it isn't required by the uniform measurement assumption. Nevertheless, the uniform measurement assumption implies that $S_{p_-}$ is crossed by a measurement at some time $m \in \{ t_k, \ldots, t_{k+1}-1\}$. Finally, we will say that $S_l$ is crossed at time $m$ if it is either crossed by an edge or crossed by a measurement (plainly or from left or right).} 

\ao{We next describe how we associate terms on the right-hand side of Eq. (\ref{claimeq}) with terms on the left-hand side of Eq. (\ref{claimeq}). 
Suppose $l$ is any number in $1, \ldots, n-1$ except $p_{-}$;  consider the {\em first} time $S_l$ is crossed; let this be time $m$. If the crossing is  by an edge,  then let $(i,j)$ be any edge which goes between $S_l$ and $S_l^c$ at time $m$. We will associate $(z_l(t_k)-z_{l+1}(t_k))^2$ \aor{on the right-hand side of Eq. (\ref{claimeq})} with $(z_i(m)-z_j(m))^2$ \aor{on the left-hand side of Eq. (\ref{claimeq})}; as a shorthand for this, we will say that we associate index $l$ with the edge $(i,j)$ at time $m$. On the other hand, if $S_l$ is crossed by a measurement\footnote{If $S_l$ is crossed both by an edge and a measurement at time $m$, we will say it is crossed by an edge first. Throughout the remainder of this proof, we keep to the convention breaking ties in favor of edges by saying that $S_l$ is crossed first by an edge if the first crossing was simultaneously by both an edge and by a measurement.} at time $m$, let $i$ be a node in whichever of $S_l, S_l^c$ does not contain zero which has a measurement at time $m$. We associate $(z_l(t_k) - z_{l+1}(t_k))^2$ with $z_i^2(m)$; as a shorthand for this, we will say that we associate index $l$ with a measurement by $i$ at time $m$. }

\ao{Finally, we describe the associations for  the terms $v_{p_-}(t_k)^2$ and $v_{p+}(t_k)^2$, which are more intricate. Again, let us suppose that $S_{p_-}$ is crossed first at time $m$; if the crossing is by an edge, then we associate both these terms with any edge $(i,j)$ crossing $S_{p_-}$ at time $m$. If, however, $S_{p_-}$ is crossed first by a measurement from the left, then we associate $v_{p_-}^2(t_k)$ with $z_i^2(m)$, where $i$ is any node in $S_{p_-}$ having a measurement at time $m$. We then consider $u$, which is the first time $S_{p_-}$ is crossed by either an edge or a measurement from the right; if it is crossed by an edge, then  we associate $v_{p_+}(t_k)^2$ with $(z_i(u)-z_j(u))^2$ with any edge $(i,j)$ going between $S_{p_-}$ and $S_{p_-}^c$ at time $u$; else, we associate it with $z_i^2(u)$ where $i$ is any node in $S_{p_-}^c$ having a measurement at time $u$.  On the other hand, if $S_{p_-}$ is crossed first by a measurement from the right, then we flip the associations:
we associate $v_{p_+}^2(t_k)$ with $z_i^2(m)$, where $i$ is any node in $S_{p_-}^c$ having a measurement at time $m$. We then consider $u$, which is now the first time $S_{p_-}$ is crossed by either an edge or a measurement from the left.  If $S_{p_-}$ is crossed by an edge first, then we associate $v_{p-}(t_k)^2$ with $(z_i(u)-z_j(u))^2$ with any edge $(i,j)$ going between $S_{p_-}$ and $S_{p_-}^c$ at time $u$; else, we associate it with $z_i^2(u)$ where $i$ is any node in $S_{p_-}$ having a measurement at time $u$. }

\aor{It will be convenient for us to adopt the following shorthand: whenever we associate $v_{p_-}^2(t_k)$ with an edge or measurement as shorthand we will say that we associate the {\em border} $p_-$, 
and likewise for $p_+$. Thus we will refer to the association of {\em indices} $l_1, \ldots, l_k$ and {\em borders} $p_-, p_+$ with the understanding that the former refer
to the terms $(v_{l_i}(t_k) - v_{l_i + 1}(t_k))^2$ while the latter refer to the terms $v_{p_-}^2(t_k)$ and $v_{p_+}^2(t_k)$.}

\ao{We now go on to prove that every term on the left-hand side of Eq. (\ref{claimeq}) is at least as big as the sum of all terms on the right-hand side of Eq. (\ref{claimeq}) associated with it. }

\ao{Let us first consider the terms $(z_i(m)-z_j(m))^2$ on the left-hand side of Eq. (\ref{claimeq}). Suppose the edge $(i,j)$ with $i<j$ at time $m$ was associated with indices $l_1 < l_2 < \cdots < l_r$.  
\aor{It must be that $i \leq l_1$ while $j \geq l_{r}+1$.} The key observation is that if any $S_l$ has not been crossed before time $m$ then  \[ \max_{i=1, \ldots, l} z_i(m) \leq z_l(t_k) \leq z_{l+1}(t_k) \leq \min_{i=l+1, \ldots, n} z_i(m). \] Consequently, \[ z_i(m) \leq z_{l_1}(t_k) \leq z_{l_1+1}(t_k) \leq z_{l_2}(t_k) \leq z_{l_2+1}(t_k) \leq \cdots \leq z_{l_r}(t_k) \leq z_{l_r+1}(t_k) \leq z_j(m) \] which implies that
\[ (z_i(m) - z_j(m))^2 \geq (z_{l_1+1}(t_k) - z_{l_1}(t_k))^2 + (z_{l_2+1}(t_k)-z_{l_2}(t_k))^2 + \cdots + (z_{l_r+ 1}(t_k)-z_{l_r}(t_k))^2. \] 
This proves the statement in the case when the edge $(i,j)$ is associated with indices
$l_1 < l_2 < \cdots < l_r$.}

\ao{Suppose now that the edge $(i,j)$ is associated with indices $l_1 < l_2 < \cdots <l_r$ as well as both borders $p_-,p_+$. This happens when
every $S_{l_i}$ and $S_{p_-}$ is crossed for the first time by $(i,j)$, so that we can simply repeat the argument in the previous paragraph
to obtain \[ (z_i(m) - z_j(m))^2 \geq (z_{l_1+1}(t_k) - z_{l_1}(t_k))^2 + (z_{l_2+1}(t_k)-z_{l_2}(t_k))^2 + \cdots + (z_{l_r+ 1}(t_k)-z_{l_r}(t_k))^2 + (z_{p_-}(t_k) \aor{-} z_{p_+}(t_k))^2 \] which, since $(z_{p_-}(t_k) - z_{p_+}(t_k))^2 \geq z_{p_-}^2(t_k) + z_{p_+}^2(t_k)$ proves the statement in this case.}

\ao{Suppose now that the edge $(i,j)$ with $i<j$ at time $m$ is associated with indices $l_1 < l_2 < \cdots < l_r$ as well as the border $p_-$. \aor{From our
association rules, this can only happen in the following case:} every $S_{l_i}$ has not been crossed before time $m$, $S_{p_-}$ is being crossed by an edge at time $m$ and has been crossed from the right by a measurement but has not been crossed from the left before time $m$, nor has it been crossed by an edge before time $m$. Consequently, in addition to the inequalities $i \leq l_1, j \geq l_r+1$ we have the additional inequalities $i \leq p_-$ while $j \geq p_+$ (since $(i,j)$ crosses $S_{p_-}$). Because $S_{l_r}$ has not been crossed before and $S_{p_-}$ has not been crossed by an edge or measurement from the left before, we have \begin{eqnarray*} 
z_i(m) & \leq & \min(z_{p_-}(t_k), z_{l_1}(t_k)) \\ 
z_j(m) & \geq & \max(0, z_{l_r+1}(t_k)) 
\end{eqnarray*} so that
\begin{eqnarray*} (z_i(m)-z_j(m))^2 \geq (z_{l_1+1}(t_k) - z_{l_1}(t_k))^2 + (z_{l_2+1}(t_k)-z_{l_2}(t_k))^2 + \cdots + (z_{l_r+ 1}(t_k)-z_{l_r}(t_k))^2 + (z_{p_-}(t_k)-0)^2(t_k)
\end{eqnarray*} which proves the statement in this case.}

\ao{The proof when the edge $(i,j)$ is associated with index $l_1 < \cdots < l_r$ and $z_{p_+}^2(t_k)$ is similar, and we omit it. Similarly, the case when $(i,j)$ is associated with just one of the borders and no indices and both borders and no indices are proved with an identical argument. Consequently, we have now proved
the desired statement for all the terms of the form  $(z_i(m)-z_j(m))^2$.}

\ao{It remains to consider the terms $z_i^2(m)$. So let us suppose that the term $z_i^2(m)$ is associated with indices $l_1 < l_2 < \cdots < l_r$ as well as possibly one of the borders $p_-, p_+$. Note that due to the way we defined the associations it cannot be associated with them both. Moreover, again due to the
association rules, we will either have 
$l_1 < \cdots < l_r < p_-$ or $p_+ \leq l_1 < \cdots < l_r$; let us assume it is the former as the proof in the latter case is similar. Since $S_{l_1}$ has 
not been crossed before, we have that 
\[ z_i(m) \leq z_{l_1}(t_k) \leq z_{l_1+1}(t_k) \leq z_{l_2}(t_k) \leq z_{l_2+1}(t_k) \leq \cdots \leq z_{l_r}(t_k) \leq z_{l_r+1}(t_k) \leq z_{p_-}(t_k) < 0  \]
 and therefore
\[ z_i^2(m) \geq (z_{l_1+1}(t_k) - z_{l_1}(t_k))^2 + (z_{l_2+1}(t_k)-z_{l_2}(t_k))^2 + \cdots + (z_{l_r+ 1}(t_k)-z_{l_r}(t_k))^2 + (z_{p_-}(t_k)-0)^2 \] which proves the result in this case. Finally, the case when $z_i(m)$ is associated with just one of the borders is proved with an identical argument. This concludes the proof.}
 \end{proof} 
 
 \smallskip
 
\ao{We are now finaly able to prove the very last piece of Theorem \ref{mainthm}, namely 
Eq. (\ref{eq:general}).}

\bigskip

\begin{proof}[Proof of Eq. (\ref{eq:general})] As in the statement of Lemma \ref{lemma:quantdecbound}, we choose $t_1=1$, and 
$t_k = (k-1) \max(T,2B)$ for $k > 1$. \ao{Observe that by a continuity argument Lemma (\ref{eq:quantdecbound}) holds even with the strict inequalities between $v_i(t_k)$ replaced with nonstrict inequalities and without the assumption that none of the $v_i(t_k)$ are \aor{equal to $\mu$}. Moreover, using the inequality \[ (v_{p_-}(t_k) - \mu)^2 + (v_{p+}(t_k) - \mu)^2 \geq \frac{(v_{p_-}(t_k) - \mu)^2 + (v_{p+}(t_k) - \mu)^2 + (v_{p_-}(t_k) - v_{p_+}(t_k))^2}{4}, \]  we have that Lemma (\ref{eq:quantdecbound}) implies that}  \[ \sum_{m=t_k}^{t_{k+1}-1} \sum_{(k,l) \in E(m)} E[ (v_k(m)-v_l(m))^2 ~|~ v(t_k)] +   \sum_{k \in S(m)} E[ (v_k(m) - \mu)^2 ~|~ v(t_k)]  \geq  \frac{1}{4} \kappa(L_n) Z(t_k).\] Because $\Delta(t)$ is decreasing and the degree of any vertex at any time is at most $d_{\rm max}$, this in turn implies \begin{footnotesize}
  \[ \sum_{m=t_k}^{t_{k+1}-1} \frac{\Delta(m)}{8} \sum_{(k,l) \in E(m)} \frac{E[ (v_k(m)-v_l(m))^2 ~|~ v(t_k)]}{\max (d_k(m), d_l(m))}  + \frac{\Delta(m)}{4  }  \sum_{k \in S(m)} E[ (v_k(m) - \mu)^2 ~|~ v(t_k)]  \geq \frac{\Delta(t_{k+1})}{32 d_{\rm max}} \kappa(L_n) Z(t_k).\] \end{footnotesize}
Now appealing to Eq. (\ref{decreaseineq}), we have
\[ E[Z(t_{k+1}) ~|~ v(t_k) ] \leq \left(1 - \frac{\Delta(t_{k+1}) \kappa(L_n)}{32 d_{\rm max} } \right) Z(t_k) +  \Delta(t_k)^2 \frac{\aor{ M \sigma^2 + n  (\sigma')^2}}{16} \max(T,2B). \] 
Now applying Corollary \ref{effdecay} as well as the bound $\kappa(L_n) \geq 1/n^2$ from Lemma \ref{lemma:snonnegativity}, we get that for 
\begin{equation} \label{finalkbound} k \geq  \left[ \frac{384 n^2 d_{\rm max} \left( 1 + \max(T,2B) \right) }{\epsilon} \ln \left( \frac{128 n^2 d_{\rm max} \left( 1 + \max(T,2B) \right) }{\epsilon} \right) \right]^{1/\epsilon} \end{equation} we will have 
\[ E[ Z(t_k) ~|~ v(1) ] \leq 25(1/16)  \frac{ M \sigma^2 + n  (\sigma')^2}{k^{1 - \epsilon}} (1+\max(T,2B))^2  32 n^2 d_{\rm max}\ln k + Z(1) e^{-(k^\epsilon - 2)/(32n^2 d_{\rm max}  \left( 1 + \max(T,2B) \right) \epsilon)}. \] Now using $t_k = (k-1) \max(T,2B) $ for $k >1$, we have \begin{small}
\[ E[ Z(t_k) ~|~ v(1) ] \leq 50 n^2 d_{\rm max} (1+\max(T,2B))^2 \frac{ M \sigma^2 + n   (\sigma')^2}{\left( 1 + t_k/\max(T,2B) \right)^{1 - \epsilon}} \ln t_k + Z(1) e^{-( \left(1 + t_k/\max(T,2B)\right)^\epsilon - 2)/(32n^2 d_{\rm max}  \left( 1 + \max(T,2B) \right) \epsilon)}. \] \end{small}

For a general time $t$, we have that as long as 
\[ t \geq \max(T,2B) +  \max(T,2B) \left[ \frac{384 n^2 d_{\rm max} \left( 1 + \max(T,2B) \right) }{\epsilon} \ln \left( \frac{128 n^2 d_{\rm max} \left( 1 + \max(T,2B) \right) }{\epsilon} \right) \right]^{1/\epsilon} \] we have that there exists a $t_k \geq t - \max(T,2B)$ with $k$ satisfying the lower bound of Eq. (\ref{finalkbound}). Moreover, the increase from this $E[Z(t_k) ~|~ v(0)]$ to $E[Z(t) ~|~ v(0)]$ is upper bounded by $\max(T,2B) (1/16) \left[  M \sigma^2 + n  (\sigma')^2\right]/t_k^{2 - 2 \epsilon}$. Thus:
\[ E[ Z(t) ~|~ v(1) ] \leq 51 n^2 d_{\rm max} (1+\max(T,2B))^2  \frac{ M \sigma^2 + n  (\sigma')^2}{\left( t/\max(T,2B)  \right)^{1 - \epsilon}} \ln t + Z(1) e^{-( \left(t/\max(T,2B)\right)^\epsilon - 2)/(32n^2 d_{\rm max}  \left( 1 + \max(T,2B) \right) \epsilon)}. \] After some simple algebra, this is exactly what we sought to prove. 

 \end{proof}

\section{Simulations\label{sec:simul}} We report here on several simulations of our learning protocol. These simulations confirm the broad outlines of the bounds we have derived; the convergence to $\mu$ takes place at a rate broadly consistent with inverse polynomial decay in $t$ and the scaling with $n$ appears to be polynomial as well. 

Figure \ref{threeplots} shows  plots of the distance from $\mu$  for the complete graph, the line graph \ao{(with one of the endpoint nodes doing the sampling)}, and the star graph \ao{(with the center node doing the sampling)}, each on $40$ nodes. These are the three graphs in the left column of the figure. We caution that there is no reason to believe these charts capture the correct asymptotic behavior as $t \rightarrow \infty$. Intriguingly, the star graph and the complete graph appear to have very similar performance. By contrast, the performance of the line graph  is an order of magnitude inferior to the performance of either of these; it takes the line graph on $40$ nodes on the order of $400,000$ iterations to reach roughly the same level of accuracy that the complete graph and star graph reach after about $10,000$ iterations. 

Moreover, Figure \ref{threeplots} also shows the scaling with the number of nodes $n$ on the graphs in the right column of the figure. The graphs show the time until $||v(t) - \mu \1||_{\infty}$ decreases below a certain threshold as a function of number of nodes. We see scaling that could plausibly be superlinear for the line graph and essentially linear for the complete graph and essentially linear for the star graph over the range shown.

\begin{figure}[h]  
\begin{center}$
\begin{array}{cc}
\includegraphics[width=2.5in]{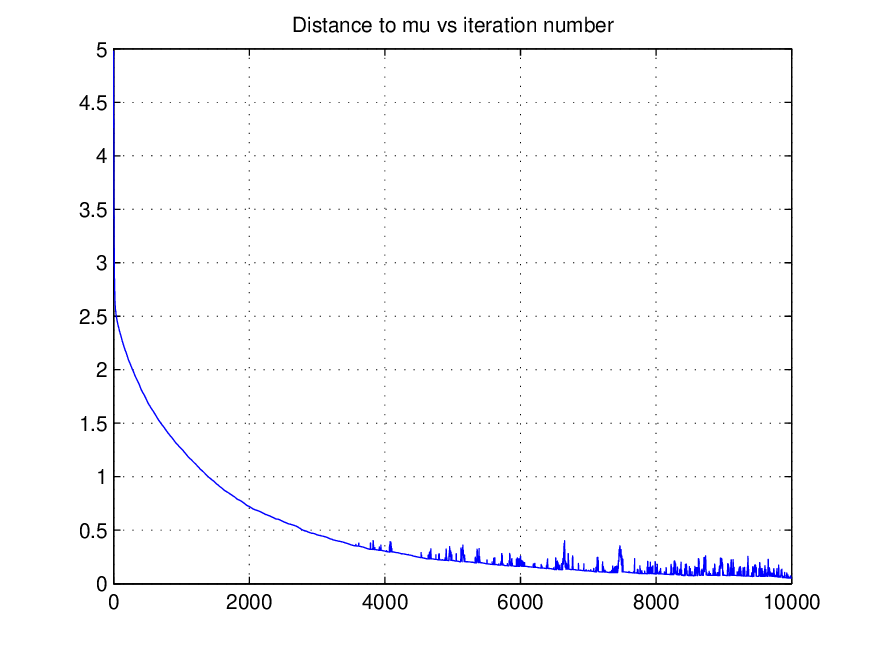} &
\includegraphics[width=2.5in]{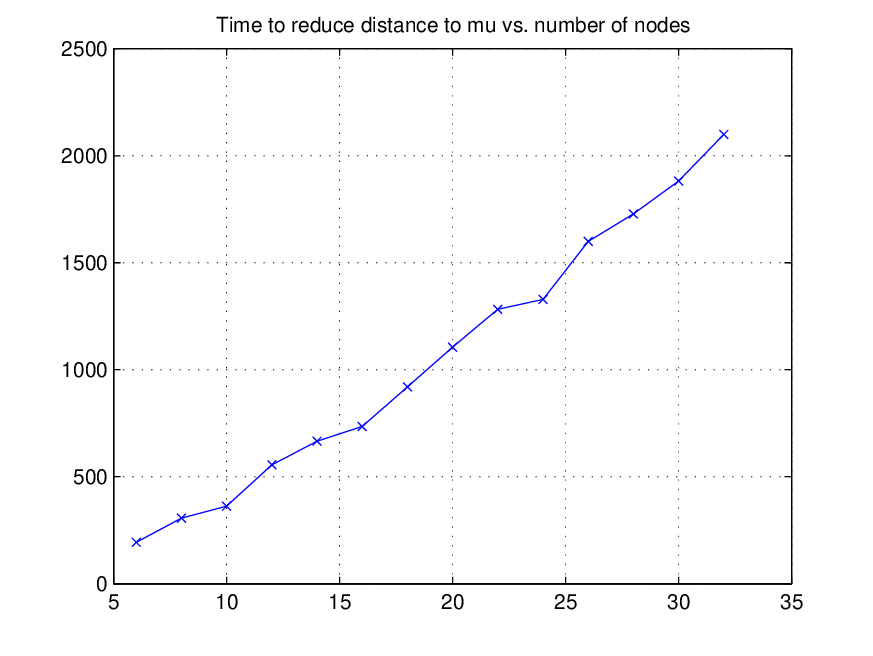}  \\ 
\includegraphics[width=2.5in]{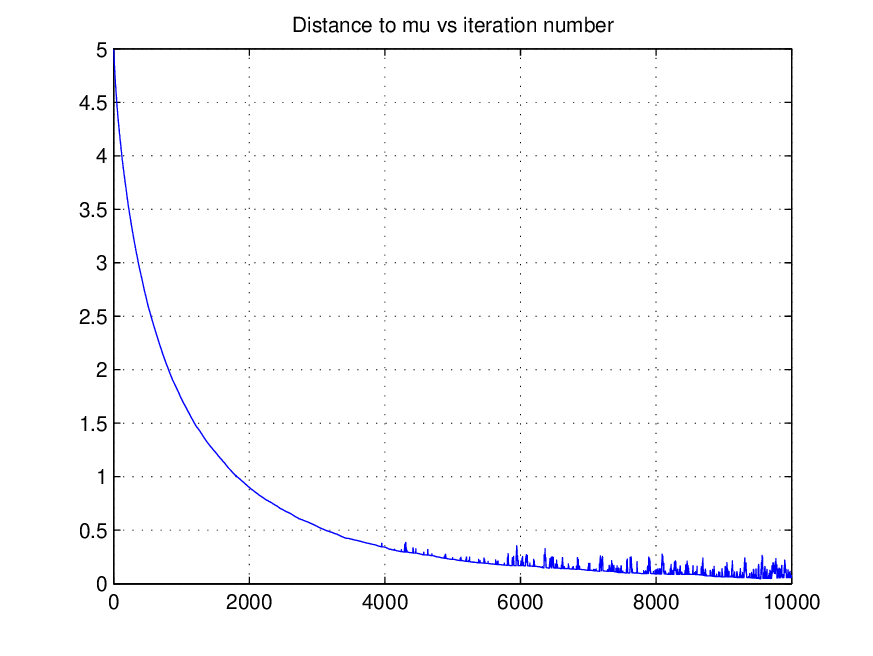} & 
\includegraphics[width=2.5in]{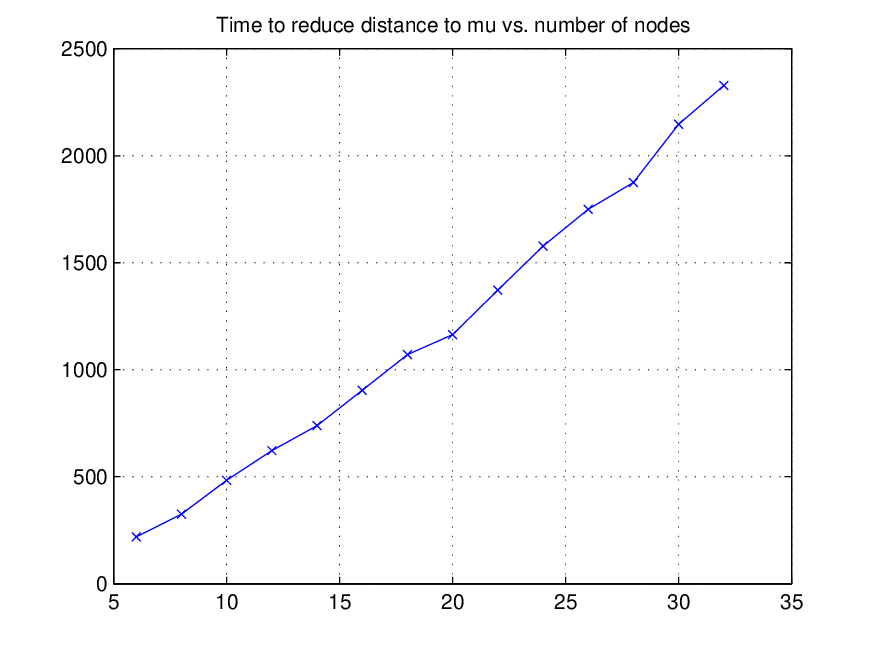}  \\
\includegraphics[width=2.5in]{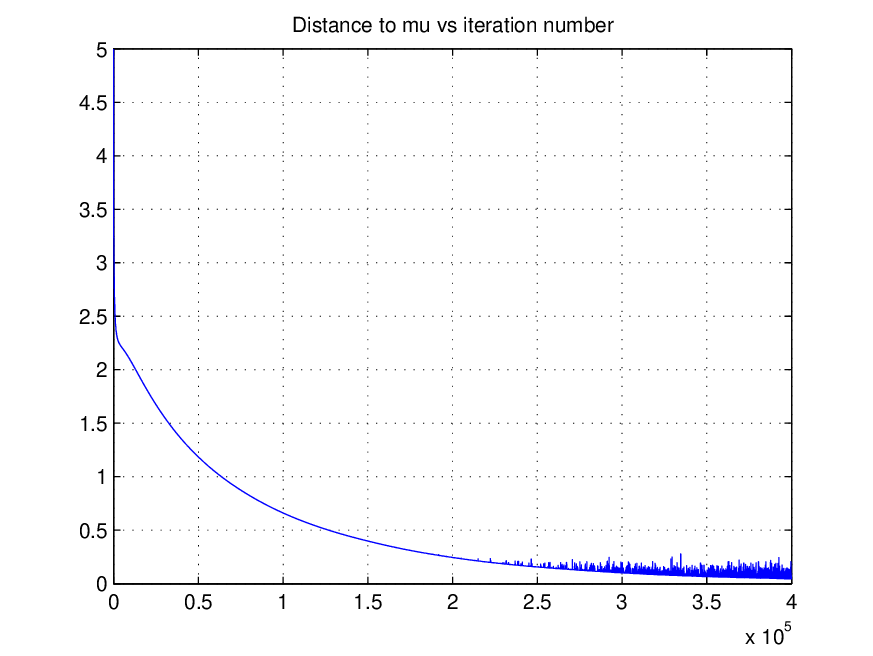}  &
\includegraphics[width=2.5in]{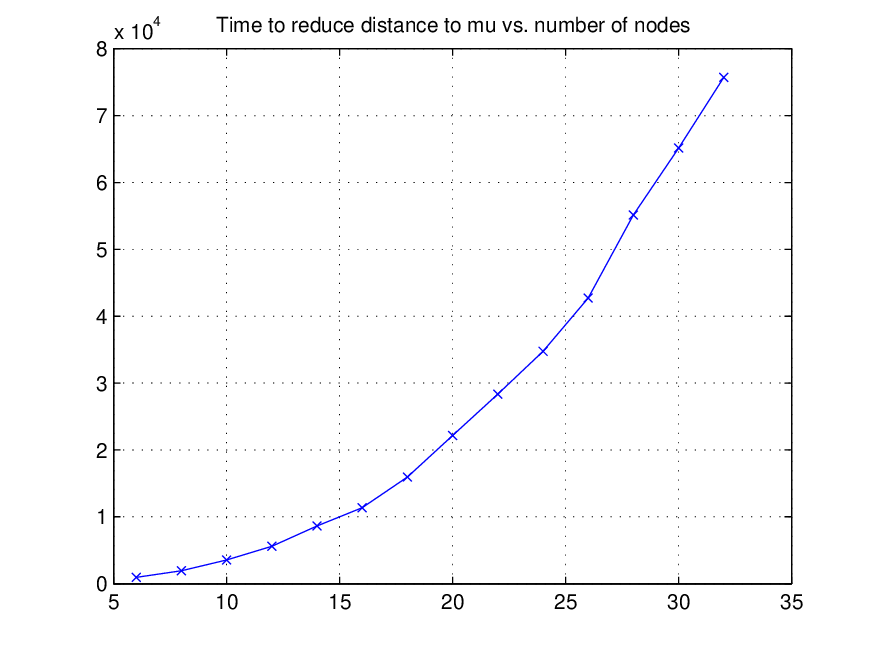} 
\end{array}$
\end{center} 
\caption{The graphs in the first column show $||v(t) - \mu \1||_{\infty}$ as a function of the number of iterations in a network of $40$ nodes starting from a random vector with entries uniformly random in $[0,5]$. The graphs in the second column show how long it takes $||v(t) - \mu \1||_{\infty}$ to shrink below $1/2$ as function of the number of nodes; inital values are also uniformly random in $[0,5]$.  The two graphs in the first row correspond to the complete graph; the  two graphs in the middle row correspond to the star graph; the two graphs in the last row 
correspond to the line graph. In each case, exactly one node is doing
the measurements; in the star graph it is the center vertex and in the line graph it is one of the endpoint vertices. Stepsize is chosen to be $1/t^{1/4}$ for all three simulations.} \label{threeplots}
 \end{figure}

Finally, we include a simulation for the lollipop graph, defined to be a complete graph on $n/2$ vertices joined to a line graph on $n/2$ vertices. The lollipop graph often appears as an extremal graph for various random walk properties (see, for example, \cite{bw90}). \ao{The node at the end of the stem, i.e., the node which is furthest from the complete subgraph, is doing the sampling.} The scaling with the number of nodes is  considerably worse than for the other graphs we have simulated here.

\begin{figure}[h] 
\begin{center}$
\begin{array}{cc}
\includegraphics[width=2.5in]{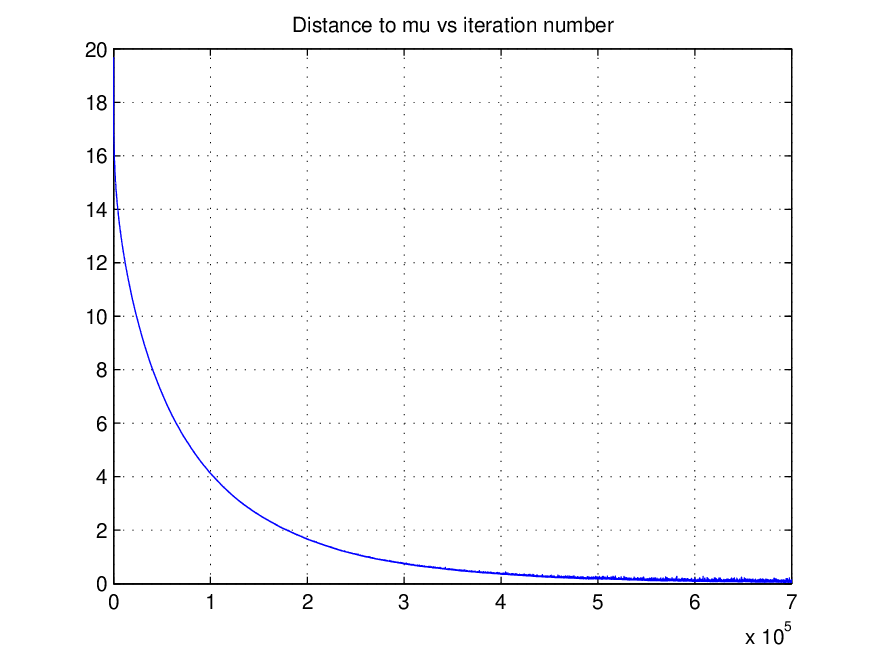} &
\includegraphics[width=2.5in]{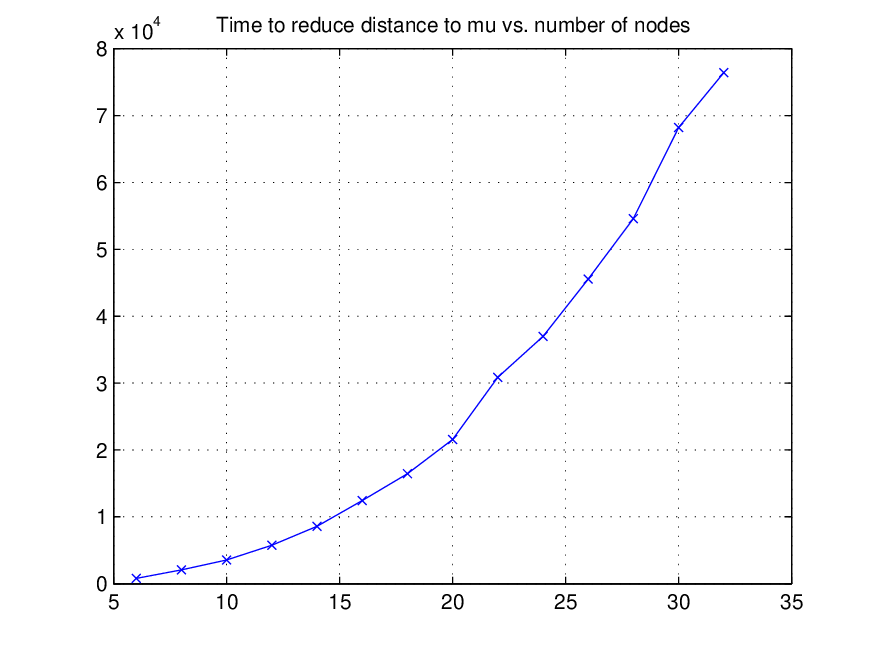} 
\end{array}$
\end{center}
\caption{The plot on the left shows $||v(t) - \mu \1||_{\infty}$ as a function of the number of iterations for the lollipop graph on $40$ nodes; the plot on the right shows the time until $||v(t) - \mu \1||_{\infty}$ shrinks below $0.5$ as function of the number of nodes $n$. In each case, exactly one node is performing the measurements, and it is the node farthest from the complete subgraph. The starting point is a random vector with entires in $[0,5]$ for both simulation and stepsize is $1/t^{1/4}$. } \label{lollipop}
\end{figure}

\ao{Finally, we emphasize that the learning speed also depends on the precise location of the node doing the sampling within the graph. While our results in this paper bound the worst case performance over all choices of sampling node, it may very well be that by appropriately choosing the sensing nodes, better performance relative to our bounds and relative to these simulations can be achieved.}

\section{Conclusion\label{sec:concl}} We have proposed a model for cooperative learning by multi-agent systems facing time-varying connectivity and intermittent measurements. We have proved a protocol capable of learning an unknown vector from independent measurements in this setting and provided quantitative bounds on its learning speed. Crucially, these bounds have a dependence on the number of agents $n$ which grows only polynomially fast, leading to reasonable scaling for our protocol.  We note that the sieve constant of a graph, a new measure of connectivity we introduced, played a central role in our analysis. On sequences of
connected graphs, the largest hitting time turned out to be the most relevant combinatorial primitive. 

Our research points to a number of intriguing open questions. Our results are for undirected graphs and it is unclear whether there is a learning protocol which will achieve similar bounds (i.e., a learning speed which depends only polynomially on $n$) on directed graphs. It appears that our bounds on the learning speed are loose by several orders of magnitude when compared to simulations, so that the learning speeds we have presented in this paper could potentially be further improved. Moreover, it is further possible that a different protocol provides a faster learning speed compared to the one we have provided here. 

Finally, and most importantly, it is of interest to develop a general theory of decentralized learning capable of handling situations in which complex concepts need to be learned by a distributed network subject to time-varying connectivity and intermittent arrival of new information. Consider, for example, a group of UAVs all of which need to learn a new strategy to deal with an unforeseen situation, for example, how to perform formation maintenance in the face \ao{of a particular pattern of turbulence}. Given that selected nodes can try different strategies, and given that nodes can observe the actions and the performance of neighboring nodes, is it possible for the entire network of nodes to collectively learn the best possible strategy? A theory of general-purpose decentralized learning, designed to parallel the theory of PAC (Provably Approximately Correct) learning in the centralized case, is warranted.  

\section{Acknowledgements} An earlier version of this paper published in the CDC proceedings had an incorrect decay rate with $t$ in the main result. The authors are grateful to Sean Meyn for pointing out this error.

\end{document}